\documentclass[a4paper,reqno]{amsart}

\usepackage{amssymb}
\usepackage{amsfonts}
\usepackage{amsmath}
\usepackage{graphicx}
\usepackage{color}
\usepackage{OD-SD-private}
\usepackage{hyperref}

\providecommand{\U}[1]{\protect\rule{.1in}{.1in}}
\numberwithin{equation}{section}
\numberwithin{equation}{section}
\numberwithin{equation}{section}
\def\e{\eps}

\newtheorem{definition}{Definition}[section]
\newtheorem{theorem}[definition]{Theorem}
\newtheorem{lemma}[definition]{Lemma}
\newtheorem{proposition}[definition]{Proposition}

\theoremstyle{definition} {\newtheorem{remark}[definition]{Remark}}

\makeindex
\begin{document}

\title{Optimal Design of Fractured Media with Prescribed Macroscopic Strain}
\author{Jos\'e Matias}
\address{CAMGSD, Departamento de Matem\'atica, Instituto Superior T\'ecnico, Av.\@ Rovisco Pais, 1, 1049-001 Lisboa, Portugal}
\email[J.~Matias]{jose.c.matias@tecnico.ulisboa.pt}
\author{Marco Morandotti}
\address{SISSA -- International School for Advanced Studies, Via Bonomea, 265, 34136 Trieste, Italy. \emph{Tel:} +39 0403787510}
\email[M.~Morandotti]{marco.morandotti@sissa.it}
\author{Elvira Zappale}
\address{Dipartimento di Ingegneria Industriale, Universit\`{a} degli Studi di Salerno, Via Giovanni Paolo II, 132, 84084 Fisciano (SA), Italy}
\email[E.~Zappale \myenv]{ezappale@unisa.it}

\date{\today. Preprint SISSA 42/2016/MATE}

\maketitle

\begin{abstract}
In this work we consider an optimal design problem for two-component fractured media for which a macroscopic strain is prescribed.
Within the framework of structured deformations, we derive an integral representation for the relaxed energy functional.
We start from an energy functional accounting for bulk and surface contributions coming from both constituents of the material; the relaxed energy densities, obtained via a blow-up method, are determined by a delicate interplay between the optimization of sharp interfaces and the diffusion of microcracks.
This model has the far-reaching perspective to incorporate elements of plasticity in optimal design of composite media.
\end{abstract}

\bigskip

\bigskip

\section{Introduction}

Starting with the pioneering papers by Kohn and Strang \cite{KS1,KS2,KS3}, much attention has been drawn to optimal design problems for mixtures of conductive materials.
The variational formulation of these problems, particularly useful for finding configurations of minimal energy, entails some technical problems from the mathematical point of view, in particular the non-existence of solutions. 
In \cite{AB, KL} this issue is addressed by introducing  a perimeter penalization in the energy functional to be minimized, which has also the effect of discarding configurations where the two materials are finely mixed. For a related problem, leading to a similar energy functional in the context of brutal damage evolution, see \cite{ABe} and \cite{FF}.


In the spirit of \cite{KS4,PS} we want to study an optimal design problem which can incorporate elements of plasticity, in a way that it is suited to treat both composite materials (made of components with different mechanical properties) and polycrystals (where the same material develops different types of slips and separations at the microscopic level).
In order to do so, we extend the framework introduced in \cite{AB}, by considering a material with two components 
each of which 
undergoes an independent (first-order) \emph{structured deformation}, according to the theory developed by Del Piero and Owen \cite{DO}. The generalization of our model to account for materials with more than two components, or to polycrystals, is straightforward.

Structured deformations set the basis to address a large variety of problems in continuum mechanics where geometrical changes can be associated with both  classical and non-classical deformations for which an analysis at macroscopic and microscopic level is required.
For instance, in a solid with a crystalline defective structure, separation of cracks at the macroscopic level may compete with slips and lattice distortions at the microscopic level preventing the use of classical theories, where deformations are assumed to be smooth.
The objective of the theory of structured deformations is to generalize the theoretical apparatus of continuum mechanics as a starting point for a unified description of 
bodies with microstructure. 
It also turns out to be relevant to describe phenomena as plasticity, damage, creation of voids, mixing, and fracture in terms of the underlying microstructure (see \cite{DO}).

\smallskip

We discuss now in more detail the application to polycrystals, which consist of a large number of grains, each having a different crystallograpic orientation, and where the intrinsic elastic and plastic response of each portion may vary from point to point. 
The anisotropic nature of crystal slip usually entails reorientation and subdivision phenomena during plastic straining of crystalline matter, even under homogeneous and gradient-free external loadings. 
This leads to spatial heterogeneity in terms of strain, stress, and crystal orientation. 
Beyond the aim of gaining fundamental insight into polycrystal plasticity, an improved understanding of grain-scale heterogeneity is important, and this is the main motivation for our work.
As noted in \cite{R2}, structural and functional devices are increasingly miniaturized. 
This involves size reduction down to the single crystal or crystal-cluster scale. 
In such parts, crystallinity becomes the dominant origin of desired or undesired anisotropy. 
In miniaturized devices plastic heterogeneity and strain localization can be sources of quality loss and failure. 
Thus, optimized design of small crystalline parts requires improved insight into crystal response and kinematics at the grain and subgrain scale under elastic, plastic, or thermal loadings.
Moreover, the better understanding of the interaction between neighbouring grains, namely the quantification of its elastoplastic interaction, is in itself relevant for the verification and improvement of existing polycrystals homogenization models. 
These models are often considered to capture the heterogeneities on material response for a polycrystals, see, e.g., \cite{R}.
In this spirit, this work can also be viewed as a first step towards the derivation of a homogenization result for a polycrystalline material in the context of plasticity.

\smallskip

To minimize our functional from a variational point of view, we rely on the energetics for structured deformations first studied by Choksi and Fonseca \cite{CF}, where the problem is set in the space of special functions of bounded variation. 
Given an open bounded subset $\Omega\subset\R N$, a structured deformation (in the context of \cite{CF}) is a pair $(g, G)\in SBV(\Omega; \R d){\times}L^1(\Omega; \R {d\times N})$, where $g$ is the microscopic deformation and $G$ is the macroscopic deformation gradient.
The energy associated with a structured deformation is then defined as the most effective way to build up the deformation using sequences $u_n \in SBV(\Omega; \R d)$ that approach $(g, G)$ in the following sense: $u_n \to g $ in $L^1(\Omega; \R d)$ and $\nabla u_n \wto G$ in $L^p(\Omega; \R {d\times N})$, for $p>1$ a given summability exponent. 
The convergences above imply that the singular parts $D^s u_n$ converge, in the sense of distributions, to $Dg-G$. 
To have a better understanding of this phenomenon, consider the simpler case of a deformation $g \in W^{1,1}(\Omega; \R d)$, that is,  without macroscopic cracks.
 Then, $Du_n = \nabla u_n \mathcal{L}^N + D^su_n$ with $D^s u_n$ absolutely continuous with respect to the Hausdorff measure $\cH^{N-1}$ and supported in $S(u_n)$, the jump set of $u_n$. Since $D u_n \wto \nabla g$ in the sense of distributions and $\nabla u_n \wto G$, we conclude that $D^s u_n \wto \nabla g - G$ in the sense of distributions. 

 This tells us that the difference between microscopic and macroscopic deformations is achieved through a limit of singular measures supported in sets $S(u_n)$ such that $\cH^{N-1}(S(u_n)) \to +\infty.$ The tensor $M:=\nabla g-G$ is called the \emph{disarrangements tensor} and embodies the fact that the difference between the microscopic and the macroscopic deformations in the bulk are achieved as a limit of singular measures.

 The results obtained in \cite{CF} show that the bulk density of the energy of a structured deformation can be influenced by both the bulk and interfacial densities of the energy of these approximating sequences, and the interplay is characterized by means of precise relations between them. 

\smallskip

The energy functional that we consider (see \eqref{energyodsd}) will feature (i) different bulk densities associated with each of the two components, 
(ii) surface energy densities to account for the jumps in the deformations inside each component, 
(iii) a perimeter penalization (which measures the boundary between the two components 
independently on the discontinuities on the deformation), and finally (iv) a surface energy term that accounts for the interaction between neighbouring components 
(where both discontinuities in the deformation and 
in the components are counted). 

More precisely, in order to take the presence of two components into account, we consider a set of finite perimeter $E\subset\Omega$, describing one of them, and let $\chi\in BV(\Omega;\{0,1\})$ be its characteristic function. 
Denoting by $\{\chi= 1\}$ the set of points in $\Omega$ with density $1$ (see \cite{AFP}), by  $\{\chi = 0\}$ the set of points in $\Omega$ with density $0$, and letting $u\in SBV(\Omega;\R d)$, we consider the following energy $F_{\odsd}\colon BV(\Omega;\{0,1\}){\times}SBV(\Omega;\R d)\to [0, +\infty[$, defined as
\begin{equation}\label{energyodsd}
\begin{split}
F_{\odsd}(\chi,u):= & \int_\Omega ((1-\chi)W^0(\nabla u)+\chi W^1(\nabla u) )\,\de x \\
+ &  \int_{\Omega\cap\{ \chi = 0\}\cap S(u)}g^0_1([u],\nu(u))\,\de\cH^{N-1}+ \int_{\Omega\cap \{ \chi = 1\}\cap S(u)} g^1_1([u],\nu(u))\,\de\cH^{N-1} \\
+ &  \int_{\Omega \cap S(\chi)\cap S(u)} g_2(\chi^+,\chi^-,u^+,u^-,\nu(u))\,\de\cH^{N-1} + |D\chi|(\Omega),
\end{split}
\end{equation}
where $\nu(u)\in\cS^{N-1}$ denotes the normal of the function $u$ to the jump set $S(u)$ of $u$, $\cS^{N-1}$ being the unit sphere in $\R N$.
For $i=0,1$, $W^i\colon\R{d{\times}N}\to\R{}$ are the bulk energy densities associated with the two components, $g_1^i\colon\R d{\times}\cS^{N-1}\to[0,+\infty[$ are the surface energy densities associated with jumps in the deformation in the two components, and $g_2\colon\{0,1\}^2{\times}(\R d)^2{\times}\cS^{N-1}\to\R{}$ is the surface energy density associated with the jumps in the deformation at the interface between the two components. 

The energy contribution of the interface, independently of the discontinuities of the deformation, is carried by $|D\chi|(\Omega)$, the total variation of $D\chi$ in $\Omega$.

In \eqref{energyodsd} we have split the jump set $S(\chi,u)$ of the pair $(\chi,u)$ into the disjoint union $S(\chi,u) = (S(\chi)\cap S(u)) \cup (S(u)\setminus S(\chi)) \cup (S(\chi)\setminus S(u))$.
In this way, we penalize the underlying structured deformation occurring in $\{\chi = 0\} \cap (S(u)\setminus S(\chi))$ and $\{\chi =1\}\cap (S(u)\setminus S(\chi))$ through $g_1^0$ and $g_1^1$, respectively, and we penalize the interface $S(\chi)$ through $1$ in $S(\chi)\setminus S(u)$ (via the perimeter term) and  through $1 + g_2$ in $S(\chi)\cap S(u)$.
Therefore, when $\chi$ jumps, we are accounting for the perimeter of $\partial E$ plus a contribution along $S(\chi)$ depending on the discontinuities of $u$.

Our main goal is to find an integral representation for the functional $\cF_{\odsd}\colon BV(\Omega;\{0,1\}){\times}SBV(\Omega;\R d){\times}\allowbreak L^1(\Omega;\R{d{\times}N})\to[0,+\infty[$ defined by
\begin{equation}\label{304bis}
\begin{split}
\cF_{\odsd}(\chi,u,G):=  \inf
\Big\{ & \liminf_{n\to\infty} F_{\odsd}(\chi_n,u_n): (\chi_n,u_n)\in BV(\Omega; \{0,1\}){\times}SBV(\Omega;\R{d}),\\
& \, \chi_n\wsto \chi \; \text{in}\; BV(\Omega; \{0,1\}), u_n\to u \text{ in }\, L^1(\Omega;\R d), \nabla u_n\wto G \, \text{in }\, L^p(\Omega;\R{d{\times}N})\Big\}.
\end{split}
\end{equation} 

Our main result (see Theorem \ref{312}) states that
for $ \chi \in BV(\Omega; \{0,1\}), \; u \in SBV(\Omega; \R d)$, $G \in L^1(\Omega; \R{d{\times}N})$, and $\mathcal{F}_{\odsd}$ defined by \eqref{304bis} for functions $W^i$, $g_1^i$, $i\in\{0,1\}$ and $g_2$ satisfying hypotheses ($H_1$)--($H_7$) in Section \ref{section:statement}, for some $p > 1$ (see Section $3$), we have that
$\mathcal{F}_{\odsd}(\chi, u, G)$ admits an integral representation of the form
\begin{equation*}
\mathcal{F}_{\odsd}(\chi,u,G)= \int_\Omega H ( \chi, \nabla u, G)\, \de x + \int_{\Omega \cap S(\chi, u)} \gamma( \chi^+, \chi^-, u^+, u^-, \nu)\, \de \cH^{N-1},
\end{equation*}
where $H$ and $\gamma$ are given in \eqref{Hsd} and \eqref{gamma}, respectively.

We observe that the bulk energy density $H\colon BV(\Omega;\{0,1\}){\times}(L^1(\Omega;\R{d{\times}N}))^2\to[0,+\infty[$ depends 
on the structured deformation on $\{\chi = 0\}$ or $\{\chi = 1\}$ (see \eqref{Hsd}) and that the interfacial energy density $\gamma\colon\{0,1\}^2{\times}(\R d)^2{\times}$ $\cS^{N-1}\to[0,+\infty[$  (see \eqref{gamma}) can be further specialized on the various pieces of the decomposition of $S(\chi, u)$, 
as noted in detail in Remark \ref{remgammaodsd}. 
We remark also that if we consider the classical deformation setting, that is no jumps in $u$ and $G = \nabla u$, then we recover an optimal design problem studied in \cite{CZ}; if we consider just one material, then we recover the results in \cite{CF}.

\smallskip
The overall plan of this work is the following: in Section \ref{section:preliminaries} we fix the notation and recall some basic results used throughout this article. 
In Section \ref{section:statement} we formulate the problem, with detailed settings and assumptions and state the main result. 
Section \ref{section:auxiliary} is devoted to proving some auxiliary results and finally we prove the main theorem in Section \ref{section:proof}.
The proof follows the blow-up method of \cite{FM}: we will compute the Radon-Nikod\'ym derivatives of $\cF_{\odsd}$ with respect to the $\cL^N$ and $\cH^{N-1}$ measures and see that they can be bounded above and below by the densities $H$ and $\gamma$, respectively.

\section{Preliminaries}\label{section:preliminaries}
In this section we fix the notation used throughout this work and give a brief survey of functions of bounded variation and sets of finite perimeter.
\subsection{Notation}
Throughout the text $\Omega \subset \R{N}$ will denote an open bounded set.

We will use the following notations:
\begin{itemize}
\item[-] ${\mathcal O}(\Omega)$ is the family of all open subsets
of $\Omega $;
\item[-] $\cM (\Omega)$ is the set of finite Radon
measures on $\Omega$;
\item[-] $\cL^{N}$ and $\cH^{N-1}$ stand for the  $N$-dimensional Lebesgue measure 
and the $\left(  N-1\right)$-dimensional Hausdorff measure in $\R N$, respectively;
\item[-] $\left |\mu \right |$ stands for the total variation of a measure  $\mu\in \cM (\Omega)$; 
\item[-] the symbol $\de x$ will also be used to denote integration with respect to $\cL^{N}$;
\item[-] $\cS^{N-1}$ stands for the unit sphere in $\R N$;
\item[-]  $Q$ denotes the unit cube of $\R3$ centered at the origin; 
\item [-] $Q^+ := Q\cap \{ x_N > 0\}$ and $Q^-$ is defined similarly;
\item[-] $Q_\eta$ denotes the unit cube of $\R N$ centered at the origin with two sides perpendicular to the vector $\eta\in\cS^{N-1}$;
\item[-] $Q(x, \delta):=x+\delta Q$, $Q_\eta(x, \delta):=x+\delta Q_\eta$; 
\item[-]  $C$ represents a generic positive constant that may change from line to line;
\item[-] $\lim_{\delta, n} := \lim_{\delta \to 0^+} \lim_{n \to \infty}, \; \lim_{k,n} := \lim_{k \to \infty} \lim_{n \to \infty}.$
\end{itemize}

\subsection{Measure Theory}

In the proof of the upper and lower bounds for the blow-up method of \cite{FM}, it is necessary to work with localizations of the functional $\cF_{\odsd}$ and see that it is a Radon measure.
The following lemma, proved in \cite{FMa}, provides sufficient conditions for a set function $\Pi: {\mathcal O}(\Omega) \rightarrow [0,+\infty)$ to be the restriction of a Radon measure on ${\mathcal O}(\Omega)$.  
It is a refinement of the De Giorgi-Letta's criterion (see \cite{DGL})  and it is of importance to apply  the Direct Method as well as for the use of relaxation methods that strongly rely on the structure of Radon measures.


\begin{lemma}[Fonseca-Mal\'y]\label{FMaly}
Let $X$ be a locally compact Hausdorff space, let $\Pi :
{\mathcal O}(X) \to [0,+\infty]$ be a set function and $\mu$ be a finite Radon measure on $X$ 
satisfying
\begin{itemize}
\item [i)] $\Pi(U) \leq \Pi (V) + \Pi(U\setminus \overline{W})$
for all $U, V, W \in {\mathcal O}(X)$ such that $W \subset\subset V \subset \subset U$;
\item[ii)] Given $U \in {\mathcal O}(X)$, for all $\varepsilon  > 0$ there exists $U_\varepsilon \in  {\mathcal O}(X)$ such that $U_\varepsilon  \subset \subset U$ and
$\Pi (U \setminus \overline{U_\varepsilon})\leq \varepsilon $.
\item[iii)]$\Pi(X)\geq \mu(X)$.
\item[iv)] $\Pi(U) \leq \mu(\overline{U})$
for all $U \in {\mathcal O}( X)$.
\end{itemize} 

Then, $\Pi = \mu\res {\mathcal O}(X)$.
\end{lemma}








\subsection{BV functions}
We start by recalling some facts on functions of bounded variation which will be
used in the sequel. We refer to \cite{AFP} and the references therein for a detailed
theory on this subject.

A function $u \in L^1(\Omega; \R d)$ is said to be of {\em bounded
variation}, and we write $u \in BV(\Omega; \R d)$, if all of its first distributional derivatives
 $D_j u_i \in \cM (\Omega)$ for  $i=1,\ldots,d$ and  $j=1,\ldots,N.$ The matrix-valued measure whose entries are $D_j u_i$ is denoted by $Du.$ The space $BV(\Omega; \R d)$ is a Banach space when endowed with the norm
$$\lVert u\rVert_{BV} := \lVert u\rVert_{L^1} + | Du|(\Omega).$$
 By the Lebesgue Decomposition theorem $Du$ can be split into the sum of two mutually singular measures $D^{a}u$ and $D^{s}u$ (the absolutely continuous part  and singular part, respectively,
of $Du$ with respect to the Lebesgue measure $\mathcal{ L}^N$). By $\nabla u$ we denote the Radon-Nikod\'{y}m derivative of $D^{a}u$ with respect to $\mathcal L^N$, so that we can write
$$Du= \nabla u \,\mathcal L^N \res \Omega + D^{s}u.$$

Let $\Omega_u$ be the set of points where the approximate limits of $u$ exists and $S(u)$ the {\it jump set} of this function, i.e., the set of
points $x\in \Omega\setminus \Omega_u$ for which there
exists $a, \,b\in \R N$ and  a unit vector $\nu  \in \S{N-1}$, normal to $S(u)$ at $x$,
such that $a\neq b$ and
\begin{equation} \label{jump1} \lim_{\eps \to 0^+} \frac {1}{\e^N} \int_{\{ y \in Q_{\nu}(x,
\e) : (y-x)\cdot\nu  > 0 \}} | u(y) - a| \, \de y = 0
\end{equation}
and
\begin{equation}\label{jump2} \lim_{\eps \to 0^+} \frac {1}{\e^N} \int_{\{ y \in Q_{\nu}(x,
\e) : (y-x)\cdot\nu  < 0 \}} | u(y) - b| \, \de y = 0.
\end{equation}
The triple $(a,b,\nu)$ is uniquely determined by (\ref{jump1}) and (\ref{jump2}) up to permutation of $(a,b)$, and a change of sign of $\nu$ and is denoted by $\left(u^+ (x),u^- (x),\nu(u) (x)\right)$.

If $u \in BV(\Omega;\R{d})$ it is well known that $S(u)$ is countably $(N-1)$-rectifiable,
i.e.
$$S(u) = \bigcup_{n=1}^{\infty}K_n \cup {\mathcal N},$$
where ${\mathcal H}^{N-1}({\mathcal N}) = 0$ and $K_n$ are compact subsets of $C^1$ hypersurfaces.
Furthermore,  ${\mathcal H}^{N-1}((\Omega\setminus \Omega_u) \setminus S(u)) = 0$ and the following decomposition holds
$$Du= \nabla u \,\mathcal L^N + [u] \otimes \nu(u) \,{\mathcal H}^{N-1}\res S(u) + D^c u,$$
where $[u]:= u^+ - u^-$ and $D^c u$ is the Cantor part of the measure $Du,$ i.e., $D^c u= D^{s}u\lfloor (\Omega_u)$.

The space of \emph{special functions of bounded variation}, $SBV(\Omega; \R d)$,
introduced by De Giorgi and Ambrosio in \cite{DGA} to study free discontinuity problems, is the space of
functions $u \in BV(\Omega; \R d)$ such that $D^cu = 0$, i.e. for which
$$ Du = \nabla u\, \cL^N + [u] \otimes \nu(u) \, \cH^{N-1} \res S(u).$$

\begin{proposition}\label{thm2.3BBBF}
If $w\in BV(\Omega;\mathbb{R}^{d})  $ then

\begin{enumerate}
\item[i)] for $\mathcal{L}^{N}$-a.e. $x\in\Omega$%
\begin{equation*}
\lim_{\varepsilon\rightarrow0^{+}}\frac{1}{\varepsilon}\left\{  \frac
{1}{\mathcal{\varepsilon}^{N}}\int_{Q\left(  x,\varepsilon\right)
}\left\vert w(y)  -w(x)  -\nabla w\left(
x\right)  \cdot( y-x)  \right\vert ^{\frac{N}{N-1}%
}\de y\right\}  ^{\frac{N-1}{N}}=0; 
\end{equation*}

\item[ii)] for $\cH^{N-1}$-a.e. $x\in S(w)$ there exist $w^{+}(x), w^{-}(x)\in\mathbb{R}^{d}$, and $\nu(x)  \in \cS^{N-1}$ normal to $S(w)$ at $x$, such that
\[
\lim_{\varepsilon\rightarrow0^{+}}\frac{1}{\varepsilon^{N}}\int_{Q_{\nu}%
^{+}\left(  x,\varepsilon\right)  }\left\vert w\left(  y\right)  -w^{+}\left(
x\right)  \right\vert \de y=0,\qquad\lim_{\varepsilon\rightarrow0^{+}}\frac
{1}{\varepsilon^{N}}\int_{Q_{\nu}^{-}\left(  x,\varepsilon\right)  }\left\vert
w\left(  y\right)  -w^{-}\left(  x\right)  \right\vert \de y=0,
\]
where $Q_{\nu}^{+}( x,\varepsilon)  :=\{y\in Q_{\nu}(x,\varepsilon) : (y-x) \cdot \nu >0\}  $ and
$Q_{\nu}^{-}( x,\varepsilon):=\{y\in Q_{\nu}(x,\varepsilon)  : (y-x)\cdot \nu <0\} $;

\item[iii)] for $\cH^{N-1}$-a.e. $x\in\Omega\setminus S(w)$%
\[
\lim_{\varepsilon\rightarrow0^{+}}\frac{1}{\mathcal{\varepsilon}^{N}}%
\int_{Q\left(  x,\varepsilon\right)  }\left\vert w(y)
-w\left(  x\right)  \right\vert \de y=0.
\]

\end{enumerate}
\end{proposition}

We next recall some properties of $BV$ functions used in the sequel. We start with the following lemma whose proof can be found in \cite{CF}:
\begin{lemma}\label{ctap}
Let $u \in BV(\Omega; \R d)$. Then there exist piecewise constant functions $u_n\in SBV(\Omega;\R d)$  such that $u_n \to u$ in $L^1(\Omega; \R d)$ and
$$| Du|(\Omega) = \lim_{n\to \infty}| Du_n|(\Omega) = \lim_{n\to \infty} \int_{S (u_n)} |[u_n](x)|\; \de{\mathcal H}^{N-1}.$$
\end{lemma}
The next result is a Lusin-type theorem for gradients due to  Alberti \cite{AL} and is essential to our arguments.
\begin{theorem}\label{Al}
Let $f \in L^1(\Omega; \R{d{\times} N})$. 
Then there exist $u \in SBV(\Omega; \R d)$ and a Borel function $g\colon: \Omega\to\R{d{\times} N}$ such that
 $$ Du = f \,{\cL}^N + g\, {\mathcal H}^{N-1}\res S(u),$$
$$ \int_{S(u)} |g| \, \de \cH^{N-1} \leq C \lVert f\rVert_{L^1(\Omega; \R{d {\times} N})}.$$
 \end{theorem}
\begin{remark}\label{Alf}
From the proof of Theorem \ref{Al} it also follows that
$$ \lVert u\rVert_{L^1(\Omega;\R d)} \leq 2 C\lVert f\rVert_{L^1(\Omega; \R{d {\times}N})}.$$
\end{remark}
\begin{lemma}[{\cite[Lemma 2.6]{FM}}]\label{lemma2.5BBBF}
Let $w \in BV(\Omega;\R d)$, for ${\mathcal H}^{N-1}$-a.e. $x$ in $S(w)$,
$$
\displaystyle{\lim_{\e \to 0^+} \frac{1}{\e^{N-1}} \int_{J_w \cap  Q_{\nu(x)}(x, \e)} |w^+(y)- w^-(y)| \,\de {\mathcal H}^{N-1} = |w^+(x)- w^-(x)|.}
$$
\end{lemma}

\subsection{Sets of finite perimeter} 
In the following we give some preliminary notions related with sets of finite
perimeter. For a detailed treatment we refer to \cite{AFP, Z}.

\begin{definition}
\label{Setsoffiniteperimeter} Let $E$ be an $\mathcal{L}^{N}$-measurable
subset of $\mathbb{R}^{N}$. For any open set $\Omega\subset\mathbb{R}^{N}$ the
perimeter of $E$ in $\Omega$, denoted by $P(E;\Omega)$, is the variation of its characteristic function 
$\chi$ in $\Omega$, i.e.
\begin{equation*}
P(E;\Omega):=\sup\left\{  \int_{E} \div\varphi \,\de x:
\varphi\in C^{1}_{c}(\Omega;\mathbb{R}^{d}), \|\varphi\|_{L^{\infty}}%
\leq1\right\}.
\end{equation*}
We say that $E$ is a set of finite perimeter in $\Omega$ if $P(E;\Omega) <+ \infty$.
\end{definition}

If $\mathcal{L}^{N}(E \cap\Omega)$ is finite, then, denoting by $\chi$ its characteristic function, we have that $\chi
\in L^{1}(\Omega)$  (see \cite[Proposition 3.6]{AFP}). It follows
that $E$ has finite perimeter in $\Omega$ if and only if $\chi \in
BV(\Omega)$ and $P(E;\Omega)$ coincides with $|D\chi|(\Omega)$, the total
variation in $\Omega$ of the distributional derivative of $\chi$.


The following approximation result can be found in \cite{BA}
\begin{lemma}\label{polyhedra}
Let $E$ be a set of finite perimeter in $\Omega$. Then, there exist a sequence of polyhedra $E_n$, with characteristic functions $\chi_n$ such that 
   $\chi_n\to \chi$ in $L^1(\Omega)$ and $P (E_n;\Omega)\to  P(E;\Omega)$.
 \end{lemma}

In the sequel we denote by $\{\chi= 1\}$ the set of points in $\Omega$ with density $1$, and by  $\{\chi = 0\}$, the set of points in $\Omega$ with density $0$.
We recall (see \cite{AFP})that for every $t \in [0,1]$, the set
$E^t$ (set of all points where $E$ has density $t$) is defined by
$$
E^t := \left\{x \in \mathbb R^N: \lim_{\varrho \to 0}\frac{\cL^N\left(E \cap B_\varrho(x)\right)}{\cL^N(B_\varrho(x))}=t\right\}.
$$
 The essential boundary of $E$ is defined by $\partial^\ast E:=\mathbb R^N \setminus (E^0\cup E^1)$, and
$$
|D \chi|(\Omega)=P(E;\Omega)={\mathcal H}^{N-1}(\Omega \cap \partial^\ast E)={\mathcal H}^{N-1}(\Omega\cap E^{\frac{1}{2}}). 
$$


\section{Statement of the problem and main results}\label{section:statement}
Let $\Omega$ be a bounded open subset of $\R{N}$ and consider continuous functions $W^i\colon \R{d{\times}N} \to[0,+\infty[$, $g^i_1\colon \R{d}\times \cS^{N-1} \to [0, +\infty[$, $i\in\{0,1\}$, and $g_2: \{0,1\}^2{\times}(\R d)^2{\times}\cS^{N-1} \to [0,+\infty[$ satisfying
\begin{enumerate}
\item[($H_1$)] \noindent there exists $c, C>0$ such that for $i\in\{0,1\}$ and $\zeta_1, \zeta_2 \in \R{d{\times}N},$
\begin{equation*}
|W^i(\zeta_1) - W^i (\zeta_2)| \leq C |\zeta_1- \zeta_2| ( 1 + |\zeta_1|^{p-1} + |\zeta_2|^{p-1}), \; \; \; \;  c|\zeta_1|^p \leq W^i(\zeta_1),
\end{equation*}
for some $p>1$;
\item[($H_2$)] there exist $c, C > 0$ such that for all $ \lambda \in \R{d}$, $\nu \in \cS^{N-1}$, and for $i\in\{0,1\},$
\begin{equation*}
c|\lambda| \leq g^i_1(\lambda, \nu)\leq C|\lambda|;
\end{equation*}
\item[($H_3$)] (positive homogeneity of degree one in the first variable) for all $t>0$, $\lambda \in \R{d}$, $\nu \in \cS^{N-1}$, and for $i\in\{0,1\}$,
\begin{equation*}
g^i_1(t\lambda, \nu) = tg^i_1(\lambda, \nu);
\end{equation*}
\item[($H_4$)] (subadditivity) for all $\lambda_1,\lambda_2 \in \R{d}$, $\nu \in \cS^{N-1}$, and for $i\in\{0,1\},$
\begin{equation*}
g^i_1(\lambda_1 + \lambda_2, \nu) \leq g^i_1(\lambda_1, \nu) + g^i_1(\lambda_2, \nu);
\end{equation*}
\item[($H_5$)] there exists $C>0$ such that for all $a,b\in\{0,1\}$, $c, d\in\R{d}$, and $\nu \in \cS^{N-1},$
\begin{equation*}
0 \leq g_2(a,b,c,d, \nu) \leq C(1+ |a - b|+|c-d|);
\end{equation*}
\item[($H_6$)] (mechanical consistency of the surface energy density) for all $a, b\in \{0,1\}$, $c,d \in \R{d}$, and $\nu \in \cS^{N-1},$
\begin{equation*}
g_2(a, b,c,d, \nu) = g_2(b,a,  d, c,-\nu);
\end{equation*}
\item[($H_7$)] there exists $C>0$ such that for all $a,b, \in \{0,1\}$, $c_i, d_i \in \R{d}$, $i=1, 2$, and $\nu \in \cS^{N-1}$,
\begin{equation*}
|g_2(a,b, c_1, d_1, \nu)- g_2(a,b, c_2, d_2, \nu)| \leq C \big||c_1-c_2|- |d_1-d_2|\big|\leq C|(c_1-d_1)-(c_2-d_2)|.
\end{equation*}
\end{enumerate}
Some comments on the hypotheses are in order.
Observe that if $g_2(a,b,c,d,\nu) = \tilde g_2(b-a, d-c, \nu)$, for some function $\tilde g_2$, then ($H_7$) corresponds to imposing Lipschitz continuity in the second variable for $\tilde g_2$.
In particular, this model includes densities of the type $g_2(a,b,c,d,\nu) = 1 + |d-c|$.

In the sequel we will use the following notation, for the sake of simplicity,
\begin{equation}\label{f}
f(\chi, \nabla u):= (1-\chi)W^0(\nabla u)+\chi W^1(\nabla u),
\end{equation}
and letting $g_1(i,\lambda,\nu):=g_1^i(\lambda,\nu)$ for every $\lambda\in\R d$, $\nu\in\cS^{N-1}$, and for $i\in \{0,1\}$, we can include all the surface energy densities in one single function $g\colon\{0,1\}^2{\times}(\R d)^2{\times}\cS^{N-1}\to[0,+\infty[$ by requiring 
\begin{equation}\label{g}
\begin{split}
&g(0, 0, u^+, u^-, \nu)= g_1(0,[u],\nu)= g_1^0([u], \nu),\\  
&g(1, 1, u^+, u^-, \nu)= g_1(1,[u],\nu )=g_1^1([u],\nu),\\
&g(\chi^+,\chi^-, u^+, u^-, \nu)= g_2(\chi^+,\chi^-, u^+, u^-, \nu), \qquad\text{for $\chi^+ \neq \chi^-$.}
\end{split}
\end{equation}
In what follows, we assume further that
\begin{equation}\label{800}
g_2(\cdot, \cdot, c, c, \cdot) = g_2(a,a,\cdot,\cdot, \cdot)=0,
\end{equation}
which is not a restriction, since $g_2$ is the density defined in $S(\chi)\cap S(u)$.


\begin{remark}\label{98}
We remark the following facts:
\begin{itemize}
\item From condition $(H_1)$, it easily follows that there exists $C > 0$ such that, for all $\zeta \in \R{d{\times}N}$ and $i \in \{0,1\}$,
$$  W^i(\zeta) \leq C( 1 + |\zeta|^p).$$
The coercivity condition on the energies $W^i$ is not physically meaningful, since the Helmholtz free energy associated with crystals may have potential wells (at matrices where the energy vanishes). 
However, it can be dropped following arguments in \cite[Proof of Prop.\@ 2.22, Step 2]{CF} that are now standard: one considers a sequence of energies $W_\eps^i(\zeta):=W^i(\zeta)+\eps|\zeta|^p$ and recovers the results for $W^i$ by letting $\eps\to0$.
\item Conditions $(H_2)$ and $(H_3)$ may also rule out some important physical settings, but they can be relaxed following arguments in \cite{CF}: the coercivity condition $(H_2)$ can be weakened by asking that the admissible sequences have bounded total variation (see \cite[pages 76 and 77 ]{CF}); the homogeneity condition $(H_3)$ can be relaxed to sublinearity $g^i_1(t\lambda, \nu) \leq tg^i_1(\lambda, \nu)$ (see \cite[last paragraph of Section 3, on page 78]{CF}).

 \item We will extend by homogeneity the functions $g_1^i$, $i\in\{0,1\}$ to all of $\R{N}$ in the second variable.
Let $\xi\in\R{N}$, then $g_1^i(\cdot,\xi):=|\xi|g_1^i(\cdot,\xi/|\xi|)$;
\item We could replace the subadditivity assumption ($H_4$) by assuming Lipschitz continuity in the first variable.
\item Without any extra difficulty one could replace $f$ in \eqref{f} by $f:T\times \R{d{\times}N}\to[0,+\infty[$, where $T$ is a set of finite cardinality of $\mathbb R^m$. We also believe that a similar analysis to the one presented below, can be performed when the range of $\chi$ is countable.
Such a case is considered, e.g., in \cite{CZ};

\end{itemize}
\end{remark}



The following remark motivates the convergences in the definition of ${\mathcal F}_{\odsd}$.
\begin{remark}[Compactness]
\label{compactness}
Assume that we have a sequence $(\chi_n, u_n)\in BV(\Omega;\{0,1\}){\times}SBV(\Omega;\R d)$ such that $u_n$ is bounded in $L^1$ and the energies 
$F_{\odsd}(\chi_n,u_n)$ are bounded. 
Then the growth assumptions $(H_1), (H_2)$, and $(H_5)$ entail that $\|\nabla u_n\|_{L^p(\Omega;\R{d{\times}N})}\leq C$, $|D u_n|(\Omega)\leq C$ and so there exist $\chi \in BV(\Omega;\{0,1\})$, $u \in SBV(\Omega;\R d)$, and $G \in L^1(\Omega;\R {d \times N})$ such that
\begin{equation*}
\begin{split}
&\chi_n \overset{\ast}{\rightharpoonup} \chi \hbox{ in } BV(\Omega;\{0,1\}), \\
&u_n \to u \hbox{ in }L^1(\Omega;\R d),\\
& \nabla u_n \wto G \hbox{ in }L^p(\Omega;\R {d\times N}).
\end{split}
\end{equation*}
\end{remark}

In order to prove our main result using the blow-up method from \cite{FM}, we address the problem of finding an integral representation for the localized functional. Given $U\in\cO(\Omega)$ and for $f$ and $g$ satisfying \eqref{f}, \eqref{g}, and \eqref{800}, define

\begin{equation*}
\begin{split}
\cF_{\text{od-sd}}(\chi,u,G;U):= & \inf_{(u_n, \chi_n)}\Bigg\{\liminf_n \int_U f(\chi_n, \nabla u_n) \,\de x \\
&+ \int_{U \cap S(u_n, \chi_n)} g(\chi_n^+, \chi_n^-, u_n^+, u_n^-,  \nu(\chi_n, u_n))\,\de\cH^{N-1} + |D\chi_n|(U):\\
&  \chi_n\in BV(U; \{0,1\}), u_n\in SBV(U;\R d),  \chi_n\wsto \chi \; \text{in}\; BV(U; \{0,1\}),  \\
& u_n\to u \text{ in } L^1(U;\R d) ,\nabla u_n\wto G \hbox{ in }L^p(U;\R {d\times N}) \; \Bigg\}.
\end{split}
\end{equation*}
Let $\mathcal{F}_{\odsd}(\chi,u,G)$ denote $\mathcal{F}_{\odsd}(\chi,u,G; \Omega)$.
Then, our main theorem reads as follows
\begin{theorem}\label{312}
Let $ \chi \in BV(\Omega; \{0,1\}), \; u \in SBV(\Omega; \R d)$, and $G \in L^1(\Omega; \R{d{\times}N})$. 
Let $\mathcal{F}_{\odsd}$ be defined by \eqref{304bis} for functions $W^i$, $g_1^i$, $i\in\{0,1\}$ and $g_2$ satisfying $(H_1)$-$(H_7)$, for some $p > 1$. 
Then $\mathcal{F}_{\odsd}(\chi, u, G)$ admits an integral representation of the form
\begin{equation}\label{intrep}
\mathcal{F}_{\odsd}(\chi,u,G)= \int_\Omega H ( \chi, \nabla u, G)\, \de x + \int_{\Omega \cap S(\chi, u)} \gamma( \chi^+, \chi^-, u^+, u^-, \nu)\, \de \cH^{N-1}
\end{equation}

where, for $i\in \{0,1\}$ and $A, B \in \R{d{\times}N}$, 
\begin{equation}\label{Hsd}
\begin{split}
H (i, A, B) =\inf_u \bigg\{ & \int_Q W^i(\nabla u)\, \de x + \int_{Q \cap S(u) } g^i_1( [u], \nu(u)) \, \de \cH^{N-1}: \\
&  u \in SBV( Q; \R{d}), \, |\nabla u| \in L^p(Q), \, u|_{\partial Q} = Ax, \, \int_Q \nabla u\,\de x = B \bigg\},
\end{split}
\end{equation}
and, for $a, b \in \{0,1\}$, $c, d \in \R{d}$, $\nu \in \cS^{N-1}$,
\begin{equation}\label{gamma} 
\begin{split}
\gamma( a, b, c, d, \nu) := \inf \bigg\{  & \int_{Q_\nu \cap \{ \chi = 0\} \cap S(u)} g^0_1([u], \nu(u))\, \de \cH^{N-1} + \int_{Q_\nu \cap \{ \chi = 1\} \cap S(u)} g_1^1([u], \nu(u)) \, \de \cH^{N-1} \\
& + \int_{Q_\nu  \cap S(\chi)\cap S(u)}  g_2(\chi^+, \chi^-,  u^+, u^-, \nu(u))\, \de \cH^{N-1} +|D \chi|(Q_\nu) : \\
& (\chi,u) \in \mathcal{A}_{\odsd}(a,b, c, d,\nu)\bigg\}
\end{split}
\end{equation}
where
\begin{equation*}
\begin{split}
\mathcal{A}_{\odsd}(a,b,c, d, \nu) := \{ & (\chi, u) \in BV(Q_\nu; \{0,1\}){\times}SBV(Q_\nu; \R{d}): \\
& \chi|_{\partial Q_\nu} =\chi_{a,b,\nu}, \;u|_{\partial Q_\nu} = u_{c,d, \nu}, \;\nabla u = 0 \;\mathcal{L}^N\text{-a.e.} \;\}
\end{split}
\end{equation*}
with 
$$ \chi_{a,b, \nu}(x) :=\begin{cases} 
	a & \text{if } x\cdot \nu > 0, \\
	b & \text{if } x\cdot \nu \leq 0,
	\end{cases}
\qquad\text{and}\quad 
u_{c,d, \nu}(x) := \begin{cases} 
	c & \text{if } x\cdot \nu > 0, \\
	d & \text{if } x\cdot \nu \leq 0.
	\end{cases}
$$

\end{theorem}


\begin{remark} \label{remgammaodsd}
The formula for $H$ only sees each component separately, so for $i \in \{0,1\}$ and $A,B\in \R{d{\times}N}$,  $H(i, A, B) = H^i_{\sd}(A,B)$, where the latter is the bulk energy density given by formula $(2.16)$ in \cite{CF}.
In particular, arguing as in \cite{CF} (see formula (4.22) therein), for $i \in \{0, 1\}$ and $A, B \in \R {d\times N}$,
\begin{equation}
\label{422CF}
H(i, A, B)  \leq C(1+|A|+|B|^p),
\end{equation}
which will be used in the sequel.
The formula for $ \gamma$ can be specialized giving rise to 3 cases:
\begin{itemize} 
\item in $S(\chi)\cap S(u)$ we have in fact the formula in its full generality, and it could be denoted as $\gamma_{\odsd}$ as it fully reflects both the optimal design and structured deformation effects.
In particular $(H_2)$ and $(H_5)$ entail that for $ a, b \in \{0,1\}, c, d, \in \R d, \nu \in \cS^{N-1}$,  
\begin{equation}
\label{gammagrowth}\gamma(a, b, c, d, \nu) \leq C( 1 + |d-c|);
\end{equation}
\item in $S(u)\setminus S(\chi)$ the formula for $\gamma$ reads as
\begin{equation*} 
\gamma_{\sd}(i,\lambda,\nu):=\inf \bigg\{ \int_{Q_\nu\cap S(v)} g^i_1 ([v], \nu(v))\, \de \cH^{N-1} , \; v \in \mathcal{A}_{\tiny{\text{sd}}}(\lambda, \nu) \bigg\}
\end{equation*}
for  $i\in\{0,1\}$, $\lambda \in\R d$ and $\nu \in \cS^{N-1}$, with
\begin{equation*} 
\mathcal{A}_{\sd}( \lambda, \nu):=\{ v \in SBV(Q_\nu; \R d):  v|_{\partial Q_\nu} = v_{\lambda, \nu}, \, \nabla v = 0 \; \mathcal{L}^N\text{-a.e.}\}
\end{equation*}
and
\begin{equation}
\label{vlambdanu}
v_{\lambda, \nu} := \begin{cases} 
	\lambda & \text{if } x\cdot \nu > 0,\\
	0 & \text{if } x\cdot \nu \leq 0.
\end{cases}
\end{equation}
The symbol $\gamma_{\sd}$ is adopted to underline that it is similar to the formula for the interfacial energy density in \cite{CF} and is due only to structured deformations.
In fact, for every  $a,b \in \{0,1\}$, $c,d \in \R d$, $\nu \in S^{N-1}$, and for $i\in\{0,1\}$, we have
\begin{equation*}
\begin{split}
\gamma(a,b,c,d,\nu) \geq & \inf \bigg\{ \int_{Q_\nu \cap \{ \chi = i\} \cap S(u)} g_1^i([u], \nu(u))\, \de \cH^{N-1}: (\chi,u) \in \mathcal{A}_{\odsd}(a,b, c, d,\nu)\bigg\} \\
= &\inf \bigg\{\int_{Q_\nu \cap S(u)} g_1^i([u], \nu(u))\, \de \cH^{N-1}: u \in {\mathcal A}_{\sd}(\lambda, \nu)\bigg\}
\end{split}
\end{equation*}
where $c-d=\lambda$. In order to prove the latter equality one can argue as in the proof of \eqref{673} below, where we show that the contribution of $g_1^i$ on $Q_\nu \cap \{\chi = j\} \cap S(u)$ with $j \in \{0,1\}, j \neq i$ is negligible.
On the other hand, when $a=b=i$, the function $\chi \equiv a$ is admissible for the definition of  $\gamma$, hence  we can conclude that 
$$
\gamma(i,i,c, d,\nu)= \gamma_{\sd}(i,\lambda,\nu).
$$
On the other hand $(H_2)$ and $(H_5)$ entail that 
there exists a constant $C>0$ such that, for every $\lambda \in \R d, \nu \in \cS^{N-1},$ 
\begin{equation}
\label{gammasdgrowth}
\gamma_{\sd}(\lambda, \nu)  \leq C|\lambda|.
\end{equation}
\item in $S(\chi)\setminus S(u)$ the formula for $\gamma$ reads
\begin{equation*} 
\gamma_{\od}(a,b,\nu) := \inf \{ |D\chi|(Q_\nu) , \; \chi \in \mathcal{A}_{\tiny{\text{od}}}(a,b, \nu)\}=|(b-a)\otimes \nu|=1
\end{equation*}
for $a,b \in \{0,1\}$ and $\nu \in \cS^{N-1}$, with
\begin{equation*} 
\mathcal{A}_{\od}( a,b, \nu):= \{ \chi \in   BV(Q_\nu; \{0,1\}):  \chi|_{\partial Q_\nu} = \chi_{a,b, \nu}= 1 \}.
\end{equation*}
We use the symbol $\gamma_{\od}$ in order to emphasize that it only reflects the optimal design setting (\cite{CZ}).
In fact, for every $c,d \in \R d$ we have
$$\gamma(a,b, c,d,\nu)\geq \inf \left\{ |D\chi|(Q_\nu) , \; \chi \in \mathcal{A}_{\tiny{\text{od}}}(a,b, \nu) \right\}\geq |(b-a)\otimes \nu|.$$
On the other hand, if $c=d$, then the function $u \equiv c$ is admissible for the definition of $\gamma( a,b,c,d, \nu)$ and we have
$$
\gamma( a,b,c,c,\nu)= \gamma_{\text{od}}(a,b, \nu)=|(b-a)\otimes \nu|,
$$
thus we can conclude that on $S(\chi)\setminus S(u)$, the surface term reduces to the perimeter  of $S(\chi)$.
\end{itemize}
\end{remark}

\section{Auxiliary Results}\label{section:auxiliary}

The following result  can be proven following arguments analogous to \cite[Proposition 3.1]{CF} and \cite[Lemma 2.20]{CF}.

\begin{lemma}\label{376}
Let $i \in \{0,1\}$ and $A$, $B \in \R{d{\times}N}$, and define
\begin{equation}\label{Hsdtilde}
\begin{split}
\tilde{H} (i, A, B) =& \inf_{u_n} \liminf_n  \bigg\{ \int_Q W^i( \nabla u_n)\, \de x + \int_{Q \cap S(u_n) } g^i_1( [u_n], \nu(u_n)) \, \de \cH^{N-1}:  \\
& u_n \in SBV( Q; \R d), \, |\nabla u_n| \in L^p(Q), \, u_n \to  Ax \, \text{in} \; L^1,  \nabla u_n \wto B \text{ in } L^p(Q;\R d) \bigg\}.
\end{split}
\end{equation}
Under the assumptions $(H_1) - (H_4)$ it results that $\tilde{H}=H$, where $H$ is the density defined in \eqref{Hsd}.
\end{lemma}

In fact, by \cite[Lemma 2.20]{CF}, $(H_1)$, and  \eqref{Hsdtilde}, we have that 

\begin{equation*}
\begin{split}
\tilde{H}(i, A, B) = \inf_{u_n}\liminf_{n\to\infty} \bigg\{ & \int_Q W^i (\nabla u_n) \de x + \int_{Q \cap S(u)} g^i_1( [u_n], \nu(u_n))\de \cH^{N-1}: \\
& u_n \in SBV( Q; \R d),u_n \to A x \hbox{ in }L^1(Q;\R d), \, \sup|u_n|_{L^\infty(Q;\R d)}<+\infty, \\
&  |\nabla u_n| \in L^p(Q;\R d),\, \nabla u_n \wto B  \hbox{ in }L^p(Q;\R d)\bigg\},
\end{split}
\end{equation*}
for $i\in \{0,1\}, A, B \in \R{d{\times}N}.$

On the other hand, exploiting the same arguments in \cite[Proposition 3.1]{CF} we have
\begin{equation*}
\begin{split} 
H(i, A, B) = \inf_u \bigg\{ & \int_Q W^i(\nabla u)\, \de x + \int_{Q \cap S(u)} g^i_1( [u], \nu(u)) \, \de \cH^{N-1}:  \\
&  u \in SBV( Q; \R{d})\cap L^\infty(Q;\R d), \, |\nabla u| \in L^p(Q), \, u\lfloor_{\partial Q} = Ax, \, \int_Q \nabla u = B \bigg\},
\end{split}
\end{equation*}
for $i\in \{0,1\}, A, B \in \R{d{\times}N}.$

Analogously, we can prove that



\begin{lemma}\label{680bis}
Assume that $(H_2)$ and $(H_4) - (H_7)$ hold.  Then, for every $a,b \in \{0,1\}$, $c,d, \in \R d$  and $\nu \in S^{N-1}$, it results

$$\gamma(a,b,c,d, \nu): = \tilde{\gamma}(a,b,c,d, \nu),$$
where
\begin{eqnarray*} \tilde{\gamma}(a,b,c,d, \nu) &=& \inf_{v_n} \left\{\liminf_{n\to \infty}  \int_{Q_\nu\cap S(\chi_n,v_n)} g (\chi^+_n, \chi^-_n, v_n^+,v_n^-,  \nu(\chi_n, v_n))\, \de \cH^{N-1}+|D \chi_n|(Q_\nu):\right. \\\\
&&\left.\chi_n \wsto \chi_{a,b,\nu} \hbox{ in }BV(Q_\nu;\{0,1\}), \;v_n \to u_{c,d, \nu} \hbox{ in }L^1(Q_\nu;\R d),\, \nabla v_n \wto 0 \hbox{ in }L^p\right\}.
\end{eqnarray*}
\end{lemma}
\begin{proof}[Proof]
Trivially we have that $\tilde{\gamma}(a,b,c,d, \nu)\leq \gamma (a,b,c,d, \nu)$. Indeed it suffices to observe that any function $u \in  SBV(Q_\nu;\mathbb R^d)$ such that $u= u_{c,d,\nu}$ on $\partial Q_\nu$ and $\nabla u=0$, and $\chi \in BV(Q_\nu;\{0,1\})$ such that $\chi= \chi_{a,b,\nu}$ on $\partial Q_\nu$ are constant sequences admissible for defining $\tilde{\gamma}$.


In order to prove the opposite inequality consider
$\nu \in S^{N-1}$, $u_n \in SBV(Q_\nu;\R d)$ and $\chi_n \in BV(Q_\nu;\{0,1\})$ such that $u_n \to u_{c,d,\nu}$ in $L^1$, with $\nabla u_n \to 0$ in $L^p$ strongly, and $\chi_n \wsto \chi_{a,b,\nu}$, with $|D \chi_n|(Q_\nu)\to |D \chi|(Q_\nu)$, i.e. $\chi_n \to \chi$ strictly. 

By Theorem \ref{Al}  we can take a sequence $v_n \in SBV(Q_\nu;\R d)$ such that $\nabla u_n=\nabla v_n$ ${\mathcal L}^{N}$-a.e. and $|D v_n|(Q_\nu)\leq C ||\nabla u_n||_{L^1(Q_\nu; \R d\times N)}$.

Then by Lemma \ref{ctap} there exist piecewise constant functions $w_{n,m}$ such that $w_{n,m}\to v_n$ as $m \to \infty$ and $|D w_{n,m}|(Q_\nu)\to |D v_n|(Q_\nu)$. 
 
 Define $z_{n,m}:=u_n-v_n+ w_{n,m}$. It results $\nabla z_{n,m}=0$ ${\mathcal L}^{N}$-a.e. Furthermore $\lim_{n,m}\|z_{n,m}-u_{c,d,\nu}\|_{L^1}= 0$. Moreover, using the fact that
 $$
 |D^s v_n|(Q_\nu)+|D^s w_{n,m}|(Q_\nu)\leq C\int_{Q_\nu}|\nabla u_n|dx \to 0 
 $$
 as $n \to\infty$ and exploiting $(H_2)$ and $(H_7)$,
 we have that
\begin{equation*}
\begin{split}
\lim_{n,m}\int_{Q_\nu \cap S(\chi_n, z_{n,m})} & g(\chi_n^+,\chi^-_n, z_{n,m}^+,z_{n,m}^-,\nu(\chi_n,z_{n,m}))\de {\mathcal H}^{N-1} \\
\leq & \lim_{n\to \infty}\int_{Q_\nu \cap S(\chi_n, u_{n})}g(\chi_n^+,\chi^-_n, u_n^+,u_n^-,\nu(\chi_n,u_n))\de {\mathcal H}^{N-1}.
\end{split}
\end{equation*}
 
Extract a diagonal sequence in $n$ and $m$, say $(\chi_k,z_k)$, such that 
 $z_k \to u_{c,d,\nu}$ in $L^1(Q_\nu)$ with $\nabla z_k =0$ ${\mathcal L}^{N}$-a.e. and $\chi_k\to \chi_{a,b,\nu}$ strictly in $BV(Q_\nu,\{0,1\})$, are such that
\begin{equation*}
\lim_{k\to \infty}\int_{Q_\nu \cap S(\chi_k, z_k)}g(\chi_k^+,\chi^-_k, z_k^+,z_k^-,\nu(\chi_k,z_{k}))\de {\mathcal H}^{N-1}\leq
\lim_{n\to \infty}\int_{Q_\nu \cap S(\chi_n, u_{n})}g(\chi_n^+,\chi^-_n, u_n^+,u_n^-,\nu(\chi_n,u_n))\de {\mathcal H}^{N-1}.
\end{equation*}
Finally we modify the sequences $z_k$ and $\chi_k$ near the boundary of $Q_\nu$. Applying Fubini's theorem we can find $r_k \to 1^-$ such that, up to a subsequence,
$$
\int_{\partial Q_\nu(0,1-r_k)} |\text{tr} \chi_k- \chi_{a,b,\nu}|\de {\mathcal H}^{N-1} \to 0,\qquad
\int_{\partial Q_\nu(0,1-r_k)} |\text{tr} z_k- u_{c,d,\nu}|\de {\mathcal H}^{N-1},
$$
as $k \to \infty$.
 
Define 
$$\tilde{\chi}_k(x):=\left\{ \begin{array}{ll}
\chi_k &\hbox{ if }x \in Q_\nu(0,1-r_k),\\
\chi_{a,b,\nu} &\hbox{ if }x \in Q_\nu(0,1)\setminus Q_\nu(0,1-r_k),
\end{array}\right.
$$
and 
$$\tilde{z}_k(x):=\left\{ \begin{array}{ll}
z_k &\hbox{ if }x \in Q_\nu(0,1-r_k),\\
u_{c,d,\nu} &\hbox{ if }x \in Q_\nu(0,1)\setminus Q_\nu(0,1-r_k).
\end{array}\right.
$$ 
Clearly $\nabla \tilde{z}_k= 0$ $ \cL^N$-a.e. and $(H_2), (H_4), (H_5), (H_7)$ ,  and the above convergences entail that
\begin{equation*}
\lim_{k\to \infty}\int_{Q_\nu \cap S(\tilde{\chi}_k, \tilde{z}_k)}g(\tilde{\chi}_k^+,\tilde{\chi}^-_k, \tilde{z}_k^+, \tilde{z}_k^-,\nu(\tilde{\chi}_k,\tilde{z}_{k}))\de {\mathcal H}^{N-1}\leq 
\lim_{n\to \infty}\int_{Q_\nu \cap S(\chi_n, u_{n})}g(\chi_n^+,\chi^-_n, u_n^+,u_n^-,\nu(\chi_n,u_n))\de {\mathcal H}^{N-1},
\end{equation*}
which concludes the proof.
\end{proof}

\begin{remark}\label{918a}
A similar argument leads to the following sequential characterization of $\gamma_{sd}$ (see Remark \ref{remgammaodsd} and Proposition $4.1$ in \cite{CF}):
\begin{equation*}
 \gamma_{\text{sd}}(i,\lambda,\nu) =  \tilde{\gamma}_{\text{sd}}(i,\lambda, \nu),
 \end{equation*}
 for every $i \in \{0,1\}$, $\lambda \in \R d$  and $\nu \in S^{N-1}$,
 where
\begin{equation*} 
\gamma_{\text{sd}}(i,\lambda,\nu) = \inf_{v_n} \bigg\{\liminf_{n\to \infty}  \int_{Q_\nu\cap S(v_n)} g_1^i ([v_n],  \nu(v_n))\, \de \cH^{N-1}:v_n \to u_{c,d, \nu} \hbox{ in }L^1(Q_\nu;\R d),\, \nabla v_n \wto 0 \hbox{ in }L^p\bigg\}.
\end{equation*}
\end{remark}
\begin{lemma}\label{Lipgammaodsd}
Let $g$ satisfy $(H_2)$, $(H_4)$, and $(H_7)$. Then 
\begin{equation}
\label{lipgamma}
| \gamma (a,b,c',d',\nu)- \gamma(a,b,c,d,\nu)|\leq C(|c-c'|+|d-d'|)
\end{equation}
for every $a,b \in \{0,1\}$, $c,c',d,d' \in \R {d}$, $\nu \in S^{N-1}$. Moreover, $\gamma$ is upper semicontinuous with respect to $\nu$.
\end{lemma}
\begin{proof}[Proof]
We start by proving \eqref{lipgamma}. By Lemma \ref{680bis}, for any given $\varepsilon >0$ there exist sequences $\chi_n \in BV(Q_\nu;\{0,1\})$ such that $\chi_n \wsto \chi_{a,b,\nu}$, and $v_n \in SBV(Q_\nu;\R d)$ such that $v_n \to u_{c,d,\nu}$ in $L^1(Q_\nu;\R d)$, $\nabla v_n \wto 0$ in $L^p(Q_\nu;\R d)$ and  
$$
\varepsilon + \gamma (a,b,c,d,\nu)\geq \lim_{n\to \infty }\int_{Q_\nu \cap S(\chi_n, v_n)} g (\chi^+_n, \chi^-_n, v_n^+,v_n^-,  \nu(\chi_n, v_n))\, \de \cH^{N-1}+|D \chi_n|(Q_\nu).
$$
By Lemma \ref{ctap} there exists a sequence of piecewise constant functions $u_n$ such that 
$$
u_n \to -u_{c,d,\nu}+u_{c', d', \nu}, \qquad |D u_n|(Q_\nu)\to |D(u_{c,d,\nu} -u_{c',d',\nu})|(Q_\nu)= |(c-c')-(d-d')|.
$$
By Lemma \ref{680bis} we have that
\begin{equation*}
\begin{split}
\gamma (a,b,c',d',\nu)&\leq {\liminf_{n\to \infty}\int_{Q_\nu \cap S(\chi_n, w_n)}g (\chi^+_n, \chi^-_n, w_n^+,w_n^-,  \nu(\chi_n, w_n))\, \de \cH^{N-1}+|D \chi_n|(Q_\nu)}\\
&\leq {\liminf_{n\to \infty}\int_{Q_\nu \cap S(\chi_n, v_n)}g (\chi^+_n, \chi^-_n, v_n^+,v_n^-,  \nu(\chi_n, v_n))\, \de \cH^{N-1}+|D \chi_n|(Q_\nu)+ \varepsilon }\\
&\leq {\gamma (a,b,c,d,\nu)+ \varepsilon + C |(c-c')|+ |(d-d')|}+\varepsilon ,
\end{split}
\end{equation*}
where  $(H_2), (H_4)$, and $(H_7)$ have been exploited. It suffices to send $\varepsilon \to 0$ to achieve one of the inequalities in \eqref{lipgamma}.
The reverse inequality can be proven in the same way.

In order to prove the upper semicontinuity of $\gamma_{\tiny\text{od-sd}}$ in the last variable, we 
observe that for every $\varepsilon >0$ there exists $\chi_\varepsilon \in BV(Q;\{0,1\})$, $\chi_\varepsilon=\chi_{a,b,\nu}$ on $\partial Q_\nu$  and $u_\varepsilon\in SBV(Q_\nu;\R d)$, $u_\varepsilon=u_{c,d,\nu}$ on $\partial Q_\nu$, with $\nabla u_\varepsilon =0$ $\cL^N$-a.e. in $Q_\nu$ and such that  
\begin{equation}
\label{gammaepsilon}
\left |\gamma (a,b,c,d,\nu)-\int_{Q_\nu \cap S(\chi_\varepsilon, u_\varepsilon)}g(\chi^+_\varepsilon, \chi^-_\varepsilon, u^+_\varepsilon, u^-_\varepsilon, \nu(\chi_\varepsilon, u_\varepsilon))\, \de {\mathcal H}^{N-1}\right |<\varepsilon.
\end{equation}

For every sequence  $\nu_n\to \nu$ we can take a family of rotations $R_n$,  such that $R_n \nu =\nu_n$ and it results clearly that $R_n$ converges to the identity. 

Then $\gamma (a,b,c,d,\nu_n)\leq \int_{Q_\nu \cap S(\chi_\varepsilon, u_\varepsilon)}g(\chi^+_\varepsilon, \chi^-_\varepsilon, u^+_\varepsilon, u^-_\varepsilon, \nu(\chi_\varepsilon, u_\varepsilon))\de {\mathcal H}^{N-1}$, which in turn, by virtue of \eqref{gammaepsilon}, provides
$$
\limsup_{n\to \infty}\gamma (a,b,c,d,\nu_n) \leq \gamma (a,b,c,d,\nu)+ \varepsilon.
$$
The proof is concluded by sending $\varepsilon \to 0$.
\end{proof}

Lemma 2.20 in \cite{CF} holds in our context leading to the following result.
\begin{lemma}\label{385}
Let  $u \in SBV(\Omega;\R d)\cap L^\infty(\Omega;\R d)$. Assume that $(H_1)-(H_7)$ hold. Then 
$$\cF_{\odsd}(\chi,u, G;U) = \cF_{\odsd}^\infty(\chi,u,G;U),$$
for every $\chi\in BV(U;\{0,1\}), G \in L^p(U;\mathbb R^d)$, and $U\in \mathcal O(\Omega)$,
where
\begin{equation*}
\begin{split}
\cF_{\odsd}^\infty(\chi,u,G;U):= & \inf_{(u_n, \chi_n)}\Bigg\{\liminf_{n\to\infty} \int_U f(\chi_n, \nabla u_n) \,\de x \\
&+ \int_{U \cap S(u_n, \chi_n)} g(\chi_n^+, \chi_n^-, u_n^+, u_n^-,  \nu(\chi_n, u_n))\,\de\cH^{N-1} + |D\chi_n|(U):\\
&  \chi_n\in BV(U; \{0,1\}), u_n\in SBV(U;\R d),  \chi_n\wsto \chi \; \text{in}\; BV(U; \{0,1\}),  \\
& u_n\to u \text{ in } L^1(U;\R d) , ||u_n||_{L^\infty} < C, \nabla u_n\wto G \hbox{ in }L^p(U;\R {d\times N}) \; \Bigg\}.
\end{split}
\end{equation*}
\end{lemma}
\begin{proof} 
In order to prove this result it clearly suffices to show that given $\chi_n\in BV(U;\{0,1\})$ such that $\chi_n  \wsto \chi$ in $BV$, and $u_n\in SBV(U;\R d)$ such that $u_n \to w_0$ in $L^1$, with $w_0 \in L^\infty(U;\R d)$ and $\nabla u_n \wto G$, in $L^p(U;\mathbb R^{d\times N})$, $|\nabla u_n| \in L^p(U)$, there exist $w_n \in SBV(U;\R d)\cap L^\infty(U;\R d)$ such that $w_n \to w_0$ in $L^1$, $\|w_n\|_{L^\infty}< \infty$ and
\begin{equation*}
\begin{split}
\liminf_{n\to\infty} \bigg\{\int_U & f(\chi_n,\nabla u_n)\,\de x +  \int_{U \cap S(u_n,\chi_n)} g(\chi^+_n, \chi^-_n, u_n^+, u_n^-, \nu(\chi_n,u_n))\,\de\cH^{N-1}\bigg\} \\
\geq & \limsup_{n\to\infty} \bigg\{ \int_U f(\chi_n,\nabla w_n)\,\de x + \int_{U \cap S(w_n,\chi_n)} g(\chi^+_n, \chi^-_n,w_n^+, w_n^-, \nu(\chi_n, w_n))\,\de\cH^{N-1}\bigg\}.
\end{split}
\end{equation*}

Let $\phi_i\in C^\infty_0(\mathbb R^d;\mathbb R^d)$ be such that \begin{equation}
\label{phii}
\phi_i(s):=\left\{ 
\begin{array}{ll}
s &\hbox{ if }|s|<e^i,\\
0 &\hbox{ if }|s|\geq e^{i+1}.
\end{array}
\right.
\end{equation}
and $\| \nabla \phi_i\|_{L^\infty}\leq 1$. 
Since $w_0\in L^\infty$, there exists $i_0$, such that for $i \geq i_0$, we have that $\|w_0\|_{L^\infty} \leq e^i$ 
and $\phi_i(w_0)=w_0 \; {\mathcal L}^N$-a.e. Let $i \geq i_0$ and define $w_n^i(x):= \phi_i(u_n(x))$, where $u_n \to w_0$ in $L^1$ and $\nabla u_n \rightharpoonup G$ in $L^p$.
Clearly $\|w^i_n\|_{L^\infty}\leq e^i$, $S(w^i_n)\subset S(u_n)$, and by the chain rule formula $\nabla w^i_n=\nabla \phi_i(u_n)\nabla u_n$ ${\mathcal L}^N$-a.e.

Furthermore, arguing as in \cite[Lemma 2.20]{CF},
$$
\|w^i_n-w_0\|_{L^1(U; \R d)}\leq \|u_n(x)-w_0(x)\|_{L^1(U; \R d)},
$$
 and $\nabla w^i_n \wto G$ in $L^p$ as $n \to \infty$.

Estimating the energies we have
\begin{equation*}
\begin{split}
\int_U &f(\chi_n,\nabla w^i_n)\,\de x + \int_{U \cap S(\chi_n,w^i_n)} g(\chi^+_n, \chi^-_n,{w_n^i}^+, {w^i_n}^-, \nu(\chi_n, w^i_n))\,\de\cH^{N-1}\\
=& \int_{\{x:|u_n|\leq e^i\}} f(\chi_n,\nabla u_n)\,\de x + \int_{\{x:e^i\leq |u_n|\leq e^{i+1}\}} f(\chi_n,\nabla \phi_i(u_n)\nabla u_n)\,\de x+ \int_{\{x:|u_n| > e^{i+1}\}} f(\chi_n,0)\,\de x\\
&+ \int_{\{x:|u_n|\leq e^i\}\cap S(\chi_n, u_n)} g(\chi^+_n, \chi^-_n, u_n^+, {u_n}^-, \nu(\chi_n, u_n))\,\de\cH^{N-1} \\
&+ \int_{\{x:e^i\leq |u_n|\leq e^{i+1}\} \cap S(\chi_n,u_n) } g(\chi^+_n, \chi^-_n,{w_n^i}^+, {w^i_n}^-, \nu(\chi_n, w^i_n))\,\de\cH^{N-1}\\
\leq & \int_U f(\chi_n,\nabla u_n)\,\de x + \int_{U \cap S(\chi_n, u_n)} g(\chi^+_n, \chi^-_n,u_n^+, u_n^-, \nu(u_n,\chi_n))\,\de\cH^{N-1}\\
&+\frac{C \|u_n\|_{L^1}}{e^{i+1}}+ C\int_{\{x:e^i\leq |u_n|\leq e^{i+1}\}} (1+|\nabla u_n|^p)dx + C\int_{ \{x:e^i\leq |u_n|\leq e^{i+1}\} \cap (S(u_n)\setminus S(\chi_n)) }|[u_n]|\de \cH^{N-1}\\
&+ C \int_{\{x:e^i\leq |u_n|\leq e^{i+1}\} \cap S(u_n)\cap S(\chi_n)}
(1+|[u_n]|)\de \cH^{N-1},
\end{split}
\end{equation*}
where we have used $(H_1)$, $(H_2)$, and the fact that ${\mathcal L}^N(\{x: |u_n|>e^{i+1}\})\leq e^{-(i+1)}\|u_n\|_{L^1}$.
Then for $M >i_0$,
\begin{equation*}
\begin{split}
&\frac{1}{M-i_0+1}\sum_{i=i_0}^M\left\{\int_U f(\chi_n,\nabla w^i_n)\,\de x + \int_{U \cap S(\chi_n,w^i_n)} g(\chi^+_n, \chi^-_n, {w_n^i}^+, {w^i_n}^-, \nu(w^i_n,\chi_n))\,\de\cH^{N-1} \right\}\\
\\
&\leq \int_U f(\chi_n,\nabla u_n)\,\de x + \int_{U \cap S(\chi_n, u_n)} g(\chi^+_n, \chi^-_n, u_n^+, u_n^-, \nu(\chi_n, u_n))\,\de\cH^{N-1}\\
\\
&+ \frac{C}{M-i_0+1}\left\{\sum_{i=i_0}^M \frac{1}{e^{i+1}} +\int_U (1+|\nabla u_n|^p)dx + \int_{U \cap S(u_n)}|[u_n]|\de \cH^{N-1}+ |D\chi_n|(U)\right\}.
\end{split}
\end{equation*}
The three terms in the last line above are uniformly bounded  independently on $n$, thus we may take $M$ so large that their sum is less than $\varepsilon$. Hence there exists some $i \in \{i_0,\dots ,M\}$ such that 
\begin{equation*}
\begin{split}
 \int_U f(\chi_n,\nabla w^i_n)\,\de x + \int_{U \cap S(\chi_n, w^i_n)} g(\chi^+_n, \chi^-_n,{w_n^i}^+, {w^i_n}^-, \nu(\chi_n, w^i_n))\,\de\cH^{N-1}\\
 \\
\leq \int_U f(\chi_n,\nabla u_n)\,\de x + \int_{U \cap S(\chi_n, u_n)} g(\chi^+_n, \chi^-_n, u_n^+, u_n^-, \nu(\chi_n, u_n))\,\de\cH^{N-1} +\varepsilon.
\end{split}
\end{equation*}
Thus it suffices to diagonalize first and then to send $\varepsilon$ to zero to obtain the result.
\end{proof}

\section{Proof of the main result}\label{section:proof}
This section is devoted to the  proof of Theorem \ref{312} and is divided in four subsections. First we prove that the functional ${\mathcal F}_{\odsd}(\chi, u, G;\cdot)$, in \eqref{304bis}, is the restriction of a suitable Radon measure to open subsets of $\Omega$, then we prove a lower bound and an upper bound estimate in terms of its integral representation when the target field $u$ is in $L^\infty(\Omega;\mathbb R^d)$ and finally we prove the general case via a truncature argument.

\subsection{Localization}

 This subsection is devoted to show that $\mathcal{F}_{\text{od-sd}}( \chi, u, G; U), \; U \in \mathcal{O}(\Omega),$ is the trace of a Radon measure absolutely continuous with respect to $\mathcal{L}^N + \cH^{N-1}\lfloor_{S(\chi,u)}$.

 \begin{proposition}
 \label{asprop2.22CF}
 Assume that $(H_1), (H_2)$, and $(H_5)$ hold and let $u \in SBV(\Omega;\R d)$. Then ${\mathcal F}(\chi, u, G; \cdot)$ is the trace on ${\mathcal O}(\Omega)$  of a finite Radon measure on ${\mathcal B}(\Omega)$.
 \end{proposition}
\begin{proof}[Proof]  The proof relies on Lemma  \ref{FMaly}.
First we prove that, for every $\chi \in BV(\Omega;\{0,1\}), u\in SBV(\Omega;\mathbb R^d)$, and $G \in L^p(\Omega;\mathbb R^{d\times N})$,  
\begin{equation*}
{\mathcal F}(\chi, u, G; U)\leq |D \chi|(U)+ {\mathcal L}^N(U)+ |Du|(U) + \|G\|_{L^p(U; \R {d \times N})}. 
\end{equation*}

We observe that by Theorem \ref{Al} there exists $h \in SBV(U;\R d)$ such that $\nabla h= G \; {\mathcal L}^N$-a.e. in $U$ and  $|D h|(U)\leq  C \|G\|_{L^1(U;\R {d \times N})}$.  By Lemma \ref{ctap} there exists a sequence of piecewise constant functions $\bar{u}_n$ such that ${\bar u_n} \to u-h $ in $L^1$, $|D {\bar u_n}|(U)\to |Du -Dh |(U)$. 

Define now $$u_n:= {\bar u_n} +h.$$
Clearly $\nabla u_n(x)= G(x)$ for ${\mathcal L}^N$-a.e. $x$ and $u_n \to u$ in $L^1$.

 Thus, the definition of ${\mathcal F}_{\odsd}(\chi, u, G; U)$, $(H_1), (H_2)$, and  $(H_5)$  entail that
\begin{align}
\label{boundmeas}
&{\mathcal F}_{\odsd}(\chi, u, G ; U)\nonumber \\
&\leq \liminf_{n\to \infty}\left\{\int_U f(\chi, \nabla u_n)\de x + \int_{U \cap S(\chi, u_n)} g(\chi^+, \chi^-, u^+_n, u^-_n, \nu(\chi, u_n))\de {\mathcal H}^{N-1}+ |D \chi|(U)\right\}\nonumber \\
&\leq \liminf_{n\to \infty}\left\{\int_U f(\chi, G )\de x + \int_{U \cap S(\chi, u_n)} C |u^+_n- u^-_n|\de {\mathcal H}^{N-1} +C|D \chi|(U)\right\}\nonumber\\
&\leq \liminf_{n\to \infty}\left\{\int_U f(\chi, G )\de x + C |D u_n|(U)+ C|D \chi|(U)\right\}\\
&\leq \liminf_{n\to \infty}\left\{\int_U f(\chi, G)\de x+ C|D \bar{u}_n|(U)+ C\|G\|_{L^1(U;\mathbb R^{d \times N})}+C|D \chi|(U)\right\}\nonumber\\
&\leq C\left\{\int_U f(\chi, G)\de x +|Du -D h|(U)+ \|G\|_{L^1(U; \R {d \times N})}+  |D \chi|(U)\right\}\nonumber \\
&\leq C\left\{\int_U f(\chi, G)\de x +|Du|(U)+ \|G\|_{L^1(U; \R {d \times N})}+  |D \chi|(U)\right\}\nonumber\\
&\leq C \left\{{\mathcal L}^N(U)+ \|G\|_{L^p(U; \R {d \times N})} + |D u|(U) +|D\chi|(U) \right\}\nonumber.
\end{align}

\noindent We start proving $(iv)$ in  Lemma \ref{FMaly}.

We know that $(H_1)$ and the lower semicontinuity of total variation entail the existence of a sequence $(\chi_n, u_n)\in BV(\Omega;\{0,1\})\times SBV(\Omega;\R d)$ such that $\chi_n \wsto \chi$ in $BV$ and $u_n \to u$ in $L^1(\Omega;\R d)$, $\nabla u_n \rightharpoonup G$ in $L^p(\Omega;\R {d \times N})$ and 
$$
{\mathcal F}_{\odsd}(\chi,u ,G; \Omega)=\lim_{n\to \infty }\left\{\int_{\Omega}f(\chi_n, \nabla u_n)\de x+ \int_{\Omega \cap S(\chi_n, u_n)}g(\chi^+_n, \chi_n^-,u^+_n,u^-_n, \nu(\chi_n, u_n))\de {\mathcal H}^{N-1}+|D \chi_n|(\Omega)\right\}.
$$
Up to the extraction of a further subsequence we know that 
$$
f(\chi_n, \nabla u_n) \de x + g(\chi^+_n, \chi^-_n, u^+_n,u^-_n, \nu(\chi_n, u_n))\de {\mathcal H}^{N-1}\lfloor_{ S(\chi_n, u_n)} + |D \chi_n|(\cdot) \wsto \mu \hbox{ in }{\mathcal M}({\overline \Omega}),
$$
as $n\to \infty$, and 
\begin{equation}
\label{2.22CF}
\mu({\overline \Omega})= {\mathcal F}_{\odsd}(\chi, u, G; \Omega) .
\end{equation}
For every $U \in \mathcal{O}(\Omega)$ we can say that
\begin{equation}\label{2.23CF}
\begin{split}
&{\mathcal F}_{\odsd}(\chi, u, G; U)\leq \liminf_{n\to \infty}\left\{\int_U f(\chi_n,\nabla u_n) \de x + \right.\\
&\left. +\int_{U \cap  S(\chi_n, u_n) }g(\chi^+_n, \chi^-_n, u^+_n, u^-_n,\nu(\chi_n, u_n))\de {\mathcal H}^{N-1}+|D \chi_n|(U)\right\}\leq \mu({\overline U}).
\end{split}
\end{equation}

Next we prove that (i) in Lemma \ref{FMaly}.

Consider $U,V,W \in {\mathcal O}(\Omega)$ such that $U\subset \subset V\subset \subset W$.
Fix $\varepsilon >0$ and consider $(\chi_n, u_n)\in BV(V;\{0,1\})\times SBV(V;\R d)$ and $(\chi'_n,v_n) \in BV(W\setminus \overline{U};\{0,1\})\times SBV(W\setminus \overline{U};\R d) $ almost minimizing sequences for ${\mathcal F}_{\text{od-sd}}$, i.e.
\begin{equation*}
\begin{split}
&\displaystyle{\lim_{n\to \infty}\left\{\int_V f(\chi_n,\nabla u_n)dx +\int_{V \cap S(\chi_n, u_n)}g(\chi^+_n,\chi^-_n,u^+_n,u^-_n, \nu(\chi_n,u_n)) \de {\mathcal H}^{N-1} + |D \chi_n|(V)\right\}}
\\
&\leq \varepsilon + {\mathcal F}_{\text{od-sd}}(\chi, u, G; V),
\end{split}
\end{equation*}
\begin{equation*}
\begin{split}
&\displaystyle{\lim_{n\to \infty}\left\{\int_{(W\setminus \overline{U})} f(\chi'_n,\nabla v_n)dx +\int_{(W\setminus \overline{U})\cap S(\chi'_n, v_n)}g({\chi'}_n^+,{\chi'}_n^-,v^+_n,v^-_n, \nu(\chi_n,v_n)) \de {\mathcal H}^{N-1} + |D \chi'_n|(W\setminus \overline{U})\right\}}\\
&\leq\varepsilon + {\mathcal F}_{\text{od-sd}}(\chi, u, G; W\setminus \overline{U}),
\end{split}
\end{equation*}
with $\chi_n \wsto \chi $ in $V$, $\chi'_n \wsto \chi$ in $W\setminus \overline{U}$, $u_n \to u$ in $L^1(V;\R d)$, $v_n \to u$ in $L^1(W\setminus \overline{U};\R d)$, $\nabla u_n \wto G$ in $L^p(V;\R {d \times N})$ and $\nabla v_n \wto G$ in $L^p(W\setminus \overline{U};\R {d\times N})$.

In order to connect the functions without adding more interfaces, we argue as in \cite{BMMO} (see also \cite{CZ}). 
For $\delta > 0$ small enough, consider
$$
U_\delta:=\{x \in V: {\rm dist}(x, \partial \overline{U}) < \delta\}.
$$
 
 For $x \in W$, let $d(x) := {\rm dist}(x;U)$. Since the distance function to a fixed set is Lipschitz
 continuous (see  \cite[Exercise 1.1]{Z}), we can apply the change of variables formula
 (see \cite[Theorem 2, Section 3.4.3]{EG}), to obtain
 $$
 \int_{U_\delta \setminus \overline{U}}|u_n(x)-v_n(x)| Jd(x)\de x=
 \int_0^\delta \left[\int_{d^{-1}(y)} |u_n(x)-v_n(x)|d {\mathcal H}^{N-1}\right]
 dy$$
 and, as since $Jd(x) $is bounded and $u_n - v_n \to 0$  in $L^1(V \cap (W \setminus \overline{U});\R d)$, it follows that for almost every $\varrho \in [0; \delta]$ we have
 
 \begin{equation}
 \label{413}
 \lim_{n\to \infty}\int_{d^{-1}(\varrho)}|u_n(x) - v_n(x)| d{\mathcal H}^{N-1}(x) = \lim_{n\to \infty}\int_{\partial U_\varrho}|u_n(x)-v_n(x)|d {\mathcal H}^{N-1}=0.
 \end{equation}
 An argument entirely analogous guarantees that
 \begin{equation}
 \label{413bis}
 \lim_{n\to \infty}\int_{d^{-1}(\varrho)}|\chi_n(x) - \chi'_n(x)| d{\mathcal H}^{N-1}(x) = \lim_{n\to \infty}\int_{\partial U_\varrho}|\chi_n(x)-\chi'_n(x)|d {\mathcal H}^{N-1}=0.
 \end{equation}
 
 Fix $\varrho_0\in [0; \delta]$ such that \eqref{413} and \eqref{413bis} hold. We
 observe that $U_{\varrho_0}$ is a set with locally Lipschitz boundary since it is a level set of a Lipschitz
 function (see e.g. Evans and Gariepy \cite{EG}). Hence we can consider $\chi_n, \chi_n', u_n,  v_n $ on $\partial U_{\varrho_0}$ in
 the sense of traces and define
 \begin{equation*}
 \chi''_n =\left\{
     \begin{array}{ll}
     \chi_n &\hbox{ in } \overline{U}_{\varrho_0}\\
     \chi'_n &\hbox{ in } W\setminus \overline{U}_{\varrho_0},
     \end{array}
     \right.
      \;\;\;\ 
  w_n =\left\{
  \begin{array}{ll}
  u_n &\hbox{ in } \overline{U}_{\varrho_0}\\
  v_n &\hbox{ in } W\setminus \overline{U}_{\varrho_0}.
  \end{array}
  \right.
 \end{equation*}

By the choice of $\varrho_0$, $\chi''_n$ and $w_n$ are admissible for ${\mathcal F}_{\odsd}(\chi, u, G, W)$.
In particular 
$$
\chi''_n \wsto \chi \hbox{ in }BV(W;\{0,1\}),
$$
$$w_n \to u  \hbox{ in } L^1(W;\R d),$$
$$
\nabla w_n \wto G \hbox{ in }L^p(W;\R {d\times N}).
$$

Thus we have
\begin{equation*}
\begin{split}
&{\mathcal F}_{\odsd}(\chi, u, G; W)\\
&\leq \liminf_{n\to \infty}\left\{\int_W f(\chi''_n, \nabla w_n)dx + \int_{W \cap S(\chi''_n, w_n)}g({\chi''}^+_n,{\chi''}^-_n,w_n^+,w_n^-, \nu(\chi''_n, w_n))\de {\mathcal H}^{N-1}+|D \chi''_n|(W)\right\}\\
&\leq \liminf_{n\to \infty} \left\{\int_V f(\chi_n, \nabla u_n)\de x + \int_{ V \cap S(\chi_n, u_n)}g(\chi^+_n,\chi^-_n,u_n^+,u_n^-, \nu(\chi_n,u_n))\de {\mathcal H}^{N-1}+|D \chi_n|(V)+\right.\\
&+\int_{W\setminus \overline{U}} f(\chi'_n, \nabla v_n)\de x + \int_{V \cap S(\chi'_n, v_n )}g(\chi'^+_n,\chi'^-_n, v_n^+,v_n^-, \nu(\chi'_n,v_n))\de {\mathcal H}^{N-1}+|D \chi'_n|(W\setminus \overline{U})+\\
&\left.+\int_{\partial U_{\varrho_0} \cap S(\chi''_n, w_n)}g({\chi''}^+_n,{\chi''}^-_n, w_n^+,w_n^-, \nu(\chi''_n,w_n))\de {\mathcal H}^{N-1} + |D \chi''_n|((S(w_n)\cup S(\chi''_n))\cap \partial U_{\varrho_0}) \right\}\\
&\leq {\mathcal F}_{\odsd}(\chi, u, G; V)+ {\mathcal F}_{\odsd}(\chi, u, G; W \setminus \overline{U})+2 \varepsilon\\
&+ \liminf_{n\to \infty}\left\{\int_{ \partial U_{\varrho_0} \cap S(\chi''_n, w_n)}g({\chi''}^+_n,{\chi''}^-_n, w_n^+,w_n^-, \nu(\chi''_n, w_n))\de {\mathcal H}^{N-1} + |D \chi''_n|(\partial U_{\varrho_0}) \cap S(\chi''_n, w_n)\right\}. 
\end{split}
\end{equation*}
Observing that, by $(H_2)$, $(H_5)$, and \eqref{413}, the first integral converges to $0$, while the convergence to $0$ of the latter term is ensured by \eqref{413bis}, the proof of  $(i)$ follows sending $\varepsilon $ to $0$.

It remains to prove $(iii)$ and $(ii)$ in Lemma \ref{FMaly}.
To this end, fix $\varepsilon >0$ and take $W\subset \subset V$ such that $\mu (V \setminus W)<\varepsilon$. By $(i)$, \eqref{2.22CF}, and \eqref{2.23CF}, it results
$$
\begin{array}{ll}
\mu(V)&\leq \mu(W)+\varepsilon 
\\
&= \mu(\overline{\Omega})-\mu(\overline{\Omega}\setminus W)+\varepsilon\\
&\leq {\mathcal F}_{\odsd}(\chi, u, G; \Omega)- {\mathcal F}_{\odsd}(\chi,u, G; \Omega \setminus \overline{W})+ \varepsilon\\
&\leq {\mathcal F}_{\odsd}(\chi, u, G; V)+\varepsilon.
\end{array}
$$
Letting $\varepsilon \to 0^+$, we obtain 
$$
\mu(V)\leq {\mathcal F}_{\odsd}(\chi,u, G; V),
$$
which proves $(iii)$.
On the other hand, by \eqref{boundmeas}, we have 
$$
{\mathcal F}_{\odsd}(\chi, u, G; \cdot)\leq C(1+ |G|^p){\mathcal L}^N+ |Du|+ |D \chi|. 
$$
Next, denote by $\lambda $ the Radon measure  on the right-hand side, take $K$ a compact set such that $K \subset \subset V$ with $\lambda(V\setminus K)<\varepsilon$, and $W$ an open set such that 
$
K\subset \subset W\subset \subset V.
$
Using $(i)$ and \eqref{2.23CF} we have
\begin{align*}
{\mathcal F}_{\odsd}(\chi, u, G; V)& \leq {\mathcal F}_{\text{od-sd}}(u,\chi, G; W)+ {\mathcal F}_{\odsd}(\chi, u, G; V\setminus K)\\
&\leq \mu(\overline{W})+ \lambda(V\setminus K)\\
&\leq \mu(V)+\varepsilon,
\end{align*}
 
 and this concludes the proof as $\varepsilon \to 0^+$.
 \end{proof}



\subsection{Lower bound }\label{lb}
This subsection is devoted to prove ``$\geq$'' in \eqref{intrep} in two steps, first identifying a lower bound for the bulk density and then for the surface one.

\subsubsection{Bulk}

Upon considering a sequence $\mu_n$ of bounded Radon measures associated with a sequence $(\chi_n,u_n)$ admissible for $\mathcal{F}_{\odsd}( \chi, u, G)$, and denoting by $\mu$ the weak-star limit of (a subsequence of)  $\mu_n$, we want to show that

$$ \frac{ d\mu}{d \mathcal{L}^N}(x _0) \geq H( \chi(x_0), \nabla u(x_0), G(x_0)),$$
for $\mathcal{L}^N$-a.e. $x_0 \in \Omega$.

Let $x_0$ be a point of absolute continuity and approximate differentiability for $\chi $ and $u$, and a point of absolute continuity  for $G$. Namely, assume that

\begin{equation}\label{440quater}
\lim_{\delta\rightarrow0^{+}}\frac{1}{\delta}\left\{  \frac
{1}{\mathcal{\delta}^{N}}\int_{Q\left(  x_0,\delta\right)
}\left\vert \chi(y)  -\chi(x_0)   \right\vert ^{\frac{N}{N-1}%
}dy\right\}  ^{\frac{N-1}{N}}=0, %
\end{equation}
\begin{equation}
\label{440bis}
\; \; \frac{d |Du|}{d\mathcal{L}^N}(x_0)= \nabla u(x_0), \, \; \frac{d |D\chi|}{d\mathcal{L}^N}(x_0)=0,
\end{equation}
\begin{equation}\label{440ter}
\lim_{\delta\rightarrow0^{+}}\frac{1}{\delta}\left\{  \frac
{1}{\mathcal{\delta}^{N}}\int_{Q\left(  x_0,\delta\right)
}\left\vert u(y)  -u(x_0)  -\nabla u\left(
x_0\right)  \cdot( y-x_0)  \right\vert ^{\frac{N}{N-1}%
}dy\right\}  ^{\frac{N-1}{N}}=0, %
\end{equation}
and
\begin{equation}\label{364}
\lim_{\delta \to 0^+} \frac{1}{\delta^N} \int_{Q(x_0, \delta)} |G(x) - G(x_0)| + |\nabla u(x) - \nabla u(x_0)| \, \de x = 0.
\end{equation}
Observe that the above requirements are satisfied for ${\mathcal L}^N$-a.e. $x_0 \in \Omega$.

\noindent Without loss of generality suppose that $\chi(x_0)=1$, the other choice can be handled similarly.

We assume that the sequence $\delta$ is chosen in such a way that $\mu(\partial Q(x_0,\delta))=0$, thus
\begin{equation*}
\begin{split}
\frac{\mu (Q(x_0,\delta))}{\delta^N} &\geq  \frac{1}{\delta^N}\liminf_{n\to \infty}\left\{\int_{x_0+\delta Q}f(\chi_n(x),\nabla u_n(x))\de x \right. \\
&\left.+  \displaystyle{ \int_{(x_0+\delta Q)\cap S(\chi_n, u_n)}
g ( \chi_n^+, \chi_n^-,u_n^+, u_n^-, \nu(\chi_n, u_n))
\de{\mathcal H}^{N-1}}
\right\}\\
&= \lim_{n\to \infty}\left\{\int_Q f(\chi_n(x_0+\delta y), \nabla u_n(x_0+\delta y))\de y \right.
\\
& \left. + \frac{1}{\delta}\int_{Q \cap \frac{S(\chi_n,u_n)-x_0}{\delta}}
g(\chi_n^+(x_0+\delta y), \chi_n^-(x_0+\delta y), u_n^+(x_0+\delta y), u_n^-(x_0+\delta y),  \nu(\chi_n,u_n)(x_0+\delta y)) \de {\mathcal H}^{N-1}\right\}.
\end{split}
\end{equation*}


Since we are estimating a lower bound, in the right hand side we can neglect the term $g_2$ in $g$, moreover, according to the notations established in \eqref{g}, $g_1(i,\cdot, \cdot)$ will denote $g_1^i(\cdot, \cdot)$, where $i \in \{0,1\}$.

Defining 

$$\chi_{n,\delta}(y):=\frac{\chi_n(x_0+\delta y)- \chi(x_0)}{\delta},$$
one has
\begin{equation}
\label{chindeltalim}
\begin{split}
\lim_{\delta,n}\|\chi_{n,\delta}\|_{L^1(Q)} = &\lim_{\delta,n}\frac{1}{\delta}\int_Q|\chi_n(x_0+\delta y)-\chi(x_0)|\de y \\
=&\lim_{\delta \to 0}\frac{1}{\delta}\int_Q|\chi(x_0+\delta y)-\chi(x_0)|\de y\\
=&\lim_{\delta \to 0}\frac{1}{\delta^{N+1}}\int_{x_0+\delta Q }|\chi(x)-\chi(x_0)| \de x =0. 
\end{split}
\end{equation}

Analogously, by defining
$$u_{n,\delta}(y):=\frac{u_n(x_0+\delta y)- u(x_0)}{\delta}$$ 
and 
$$w_0(y):=\nabla u(x_0) y,$$ 
it easily follows that
\begin{equation}\label{undeltalim}
\begin{split}
\lim_{\delta,n}\|u_{n,\delta} -w_0\|_{L^1(Q; \R d)} &=\lim_{\delta \to 0}\frac{1}{\delta}\int_{Q}|u(x_0+\delta y)-u(x_0)-\delta \nabla u(x_0)y|\de y  \\
& =\lim_{\delta\to 0}\frac{1}{\delta^{N+1}}\int_{Q}|u(x)-u(x_0)-\nabla u(x_0)(x-x_0)|\de x =0.
\end{split}
\end{equation}
Moreover $\nabla u_{n,\delta}(y)= \nabla u_n(x_0+\delta y)$.

We have that
$$ \frac{ d\mu}{d \mathcal{L}^N}(x _0) =  \lim_{k, n} \frac{\mu_n }{\delta_k^N}(Q(x_0, \delta_k)),$$
for a sequence of sides lengths $\delta_k \to 0^+$ as $k\to \infty$, and choose this sequence so that 
\begin{equation}\label{465}
\lim_{k, n} \frac{\cH^{N-1}(S(\chi_n)\cap Q(x_0, \delta_k))}{\delta_k^N} = 0.
\end{equation} 
In fact, since
\begin{equation*} 
\begin{split}
 \lim_{k, n} \frac{\cH^{N-1}(S(\chi_n)\cap Q(x_0, \delta_k))}{\delta_k^N}  \leq &  \lim_{k, n} \frac{\cH^{N-1}(S(\chi_n)\cap \overline{Q}(x_0, \delta_k))}{\delta_k^N}\\
 =& \lim_{k, n} \frac{|D\chi_n|(\overline{Q}(x_0, \delta_k))}{\delta_k^N} \leq \lim_k \frac{|D\chi|(\overline{Q}(x_0, \delta_k))}{\delta_k^N},
\end{split}
\end{equation*}
for \eqref{465} to hold, it is enough to choose $\delta_k$ so that
\begin{equation*}
\lim_k \frac{|D\chi|(\overline{Q}(x_0, \delta_k))}{\delta_k^N} = \lim_k \frac{|D\chi|(Q(x_0, \delta_k))}{\delta_k^N} = \frac{d |D\chi|}{d\mathcal{L}^N}(x_0) = 0,
\end{equation*}
where the last equality holds since  $x_0 \notin S(\chi)$.

Consequently we may estimate from below $\frac{d\mu}{d\mathcal{L}^N}(x_0)$ as
\begin{eqnarray*}
 && \liminf_{k,n}\left\{\int_Q f(\chi(x_0)+\delta_k \chi_{n,k}(y),\nabla u_{n,k}(y)) \de y \right.\\
&& \left.+ \frac{1}{\delta_k}\int_{Q \cap \frac{(S(u_{n,k})\setminus S(\chi_{n,k}))-x_0}{\delta_k}}g_1 (\chi(x_0)+\delta_k \chi_{n,k}(y),[\delta_k u_{n,k}(y)],  \nu_{n,k}(y)) \de {\mathcal H}^{N-1}\right\}\\
&&= \liminf_{k,n}  \{I_{n,k}^1 + I_{n,k}^2\},\\
\end{eqnarray*}
where we wrote for simplicity $\chi_{n,k}:= \chi_{n, \delta_k}$ and $u_{n,k}:= u_{n, \delta_k}$, and $\nu_{n,k}$ denotes the unit exterior normal to $S(u_{n,k})$.

A diagonalization argument allows to define subsequences (not relabelled) $\delta_k$, $\chi_k  := \chi_{n_k, \delta_k}$, and $u_k := u_{n_k, \delta_k}$,  such that 

\begin{equation}\label{radii}
\begin{split}
& \lim_{\delta_k \to 0}\| \chi_{k}-\chi(x_0)\|_{L^1(Q)}=0,\\
&\lim_{\delta_k \to 0} \frac{|D\chi_{k}|Q(x_0,\delta_k)}{\delta_k^N}=0,\\
& \lim_{\delta_k \to 0} \|u_{k}- w_0\|_{L^1(Q,\R d)}=0,
\end{split}
\end{equation}
and 
\begin{equation*}
\begin{split}
\frac{d\mu}{d\mathcal{L}^N}(x_0)\geq & \liminf_{k}\bigg\{  \int_Q f(\chi(x_0)+\delta_k \chi_{k}(y),\nabla u_{k}(y)) \de y \\
&+ \frac{1}{\delta_k}\int_{Q \cap \frac{(S(u_{k})\setminus S(\chi_{k}))-x_0}{\delta_k}}g_1(\chi(x_0)+\delta_k \chi_{k}(y),\delta_k[u_{k}](y), \nu_k(y)) \de {\mathcal H}^{N-1}\bigg\}\\
=& \liminf_{k}  \{I_k^1 + I_k^2\},
\end{split}
\end{equation*}

where, as above, $\nu_k$ denotes the unit normal to $S(u_k)$, and $I_k^1$ and $I^2_k$ denote $I^1_{n_k,\delta_k} $ and $I^2_{n_k,\delta_k}$, respectively.

Without loss of generality, up to subsequences if necessary, the above liminf is a limit, and, by \eqref{undeltalim} and Lemma \ref{680bis} applied to $Q$ and to $w_0:=\nabla u(x_0)\cdot y$, we can assume that $u_k$ is uniformly bounded in $L^\infty$.

We aim to fix $\chi(x_0)$ and to estimate $ \lim_{k} [ I^1_k + I^2_k]$ from below with a sequence that satisfies the conditions in the definition of $\tilde{H}(\chi(x_0), \nabla u(x_0), G(x_0))$ (see Lemma \ref{376}). For the sake of exposition, we control each term of the sum $I_{k}^1+ I^2_k$ separately and then add them. First we consider $I^1_k$.

\noindent Chacon biting Lemma (\cite[Lemma 5.32]{AFP})
guarantees the existence of a not relabelled subsequence $u_k$ and of a decreasing sequence of Borel sets $E_r$, such that ${\mathcal L}^N(E_r)\to 0$, as $r \to \infty$ and the sequence $|\nabla u_k|^p$ is equiintegrable in $Q\setminus E_r$ for any $r \in \mathbb N$.


Since $f \geq 0$ and by $(H_1)$,
\begin{eqnarray*} &&\lim_k \int_Q f( \chi(x_0) + \delta_k\chi_k(y), \nabla u_k(y)) \, \de y  \geq  \lim_k \int_{Q \setminus E_r} f( \chi(x_0) + \delta_k\chi_{k}(y), \nabla u_{k}(y)) \, \de y\\\\
&& \geq \lim_k \left\{ \int_{Q\setminus E_r} f( \chi(x_0), \nabla u_{k}(y))\, \de y - \int_{Q\setminus E_r}| \chi_{k}(x_0 + \delta_k y) - \chi(x_0)|C( 1 + |\nabla u_k(y)|^p)\, \de y \right\}\\\\
&& \geq \lim_k  \int_{Q \setminus E_r} f( \chi(x_0), \nabla u_{k}(y))\, \de y - \limsup_k C\int_{Q\setminus E_r}| \nabla u_k|^p\, \de y 
\end{eqnarray*}

where \eqref{chindeltalim} has been used.
In order to pass from $\int_{Q\setminus E_r} f(\chi(x_0), \nabla u_k(y))\de y$  to $\int_{Q} f(\chi(x_0), \nabla u_k(y))\de y$, we extract a further subsequence. 

Indeed, we claim that for each $j \in \mathbb N$ there exists $k=k(j)$ and $r_j \in \mathbb N$, such that
\begin{equation} \label{489} 
\int_{Q\setminus E_{r_j}} f(\chi(x_0), \nabla v_j(y))\, \de y \geq \int_Q f(\chi(x_0), \nabla v_j(y))\; \de y - \frac{C}{j},
\end{equation}
where $v_j := u_{k(j)}.$
In light of $(H_1)$, in order to guarantee that \eqref{489} holds, we need to make sure that
for each $j$, there exists $k=k(j)$ and $r(j)$, such that
\begin{equation*}
\int_{E_{r_j}} ( 1 + |\nabla u_{k(j)}|^p)\; \de y \leq \frac{1}{j}.
\end{equation*}
Suppose not. Then, there exists $j_0$ such that, for all $r$ and $k$, 
\begin{equation} \label{497}  
\int_{E_r} ( 1 + |\nabla u_{k} |^p)\; \de y  > \frac{1}{j_0}.
\end{equation}
For $k$ fixed, and for $r \in \N{}$ noting that $w_r = u_{k}$ is a constant sequence (and hence with $p$-equiintegrable gradients), letting $r \to \infty$ we get a contradiction from \eqref{497}.
Therefore, by 
\eqref{489}, the sequence $v_j$ gives the right estimate from below for $I^1_{k(j)}$, that is, up to the extraction of a further subsequence and denoting in what follows $\chi_j:= \chi_{k(j)},\delta_j := \delta_{k(j)}$ and $ E_j:=E_{r_j}$, we have

\begin{eqnarray*}
&&\frac{d\mu}{d\mathcal{L}^N}(x_0)\geq \lim_{j}\left\{\int_Q f(\chi(x_0),\nabla v_j(y)) \de y \right.\\
&& \left.+ \int_{Q \cap (S(v_{j})\setminus S(\chi_{k(j)})}g_1 (\chi(x_0)+\delta_{k(j)} \chi_{j}(y),[ v_{j}](y),  \nu (\chi_j, v_j)(y)) \de {\mathcal H}^{N-1}\right\}-\limsup_{j}\left\{\frac{C}{j} +\int_{Q\setminus E_{j}} |\nabla v_j|^p \de y \right\}\\
&&= \lim_{j}\left\{\int_Q f(\chi(x_0),\nabla v_j(y)) \de y +\int_{Q\cap (S(v_j)\setminus S(\chi_j))} g_1(\chi(x_0)+\delta_j \chi_j(y), [v_j](y),  \nu(\chi_j,v_j)(y)) \de {\mathcal H}^{N-1}\right\}\\
&&- \limsup_{j}\left\{\frac{C}{j} +\int_{Q\setminus E_{j}} |\nabla v_j|^p \de y \right\},
\end{eqnarray*}

where the positive $1$-homogeneity of $g_1$ in the first variables has been exploited.

Recall also that, by the choice of the sizes of the cubes in \eqref{radii}, we are going to neglect the contribution supported in $S(\chi_j)$. 

This sequence still needs to be slightly changed in order to control the surface term $I_k^2$ and to comply with the conditions in Lemma \ref{376}.




 Set now: $$F_{j} := \{y \in Q: x_0+\delta_j y \not \in S(\chi_j) \hbox{ and }|\chi_{j}(x_0 + \delta_{j} y ) - \chi(x_0)| = 0\}$$
 and note that \begin{equation}\label{511}\mathcal{L}^N(Q\setminus F_j) \to 0.\end{equation}

 Define:
 $$ \tilde{v}_j := \left\{ \begin{array}{ll} v_j & \text{in}\; F_j\\
 \nabla u(x_0)y & \text{in}\; Q\setminus F_j.\\
\end{array}
 \right.
 $$
Note that $\tilde{v}_j$ is still uniformly bounded in $L^\infty$ and it converges in $L^1$ norm to $w_0(y) = \nabla u(x_0)y.$
 Moreover, since $\nabla v_j \wto G(x_0)$ in $L^p$, by \eqref{511} the same holds for $\nabla \tilde{v}_j$.
 Next we show that passing from $v_j$ to $\tilde{v}_j$ there is no change in the control from below of $I^1_j$.
 By $(H_1)$ we have that:
 \begin{eqnarray*} 
 &&\int_Q f(\chi(x_0), \nabla v_j(y))\, \de y
\geq  
\int_Q f(\chi(x_0), \nabla \tilde{v}_j(y))\, \de y - C\int_{Q\setminus F_j} 1 + | \nabla v_j(y)|^p + |\nabla u(x_0)|^p \, \de y.\\\\
 \end{eqnarray*}
 
Since,  by \eqref{489} it results
\begin{equation*}
\begin{split}
\int_{Q\setminus F_j} |\nabla v_j(y)|^p \, \de y =& \int_{(Q\setminus {E_{j}}) \setminus F_j} |\nabla v_j(y)|^p \, \de y + \int_{ E_{j} \setminus F_j}  |\nabla v_j(y)|^p \, \de y\\
\leq &  \int_{Q \setminus E_{j}} |\nabla v_j(y)|^p \, \de y + \int_{ E_{j}\setminus F_j}  |\nabla v_j(y)|^p \, \de y\\
\leq & \int_{Q \setminus E_{j}} |\nabla v_j(y)|^p \, \de y + \frac{1}{j},
\end{split}
\end{equation*}
we obtain
 \begin{eqnarray*}
 &&\frac{d\mu}{d\mathcal{L}^N}(x_0)\geq \\
  && \lim_{j}\left\{\int_Q f(\chi(x_0),\nabla \tilde{v}_j(y)) \de y +
 \int_{Q \cap (S(v_{j})\setminus S(\chi(x_0)+\delta_j \chi_j))}g_1 (\chi(x_0)+\delta_j \chi_j(y), [v_j](y), \nu(v_j(y))) \de {\mathcal H}^{N-1}\right\}\\
 && - \limsup_{j}\left\{\frac{C}{j} +2\int_{Q\setminus E_{j}} |\nabla v_j|^p \de y \right\}
 \\
  && = \lim_{j}\left\{\int_Q f(\chi(x_0),\nabla \tilde{v}_j(y)) \de y +
\int_{Q \cap (S(v_{j})\setminus S(\chi(x_0)+\delta_j \chi_j))}g_1 ( \chi(x_0)+\delta_j \chi_j(y),[v_j],  \nu(v_j(y))) \de {\mathcal H}^{N-1}\right\}
 \end{eqnarray*}
 where the last equality follows by the equiintegrability of $|\nabla \tilde{v}_j|^p$  in $Q\setminus E_{j}$ .

Now we control $I^2_k$, observing that
\begin{eqnarray*}  
&&\liminf_{j}\int_{Q\cap S(v_j)\setminus S(\chi_j)} g_1(\chi(x_0)+\delta_j \chi_{j}(y),[v_j](y),  \nu(v_j(y)) \de {\mathcal H}^{N-1}\\\\
&& \geq \liminf_{j}\int_{Q\cap S(\tilde{v}_j)\setminus S(\chi_j)} g_1(\chi(x_0)+\delta_j \chi_{j}(y),[\tilde{v}_j](y), \nu(\tilde{v}_j(y))) \de {\mathcal H}^{N-1}\\\\
&& \geq \liminf_{j}\int_{F_j\cap S(\tilde{v}_j)} g_1^1([\tilde{v}_j](y),  \nu(\tilde{v}_j(y))) \de {\mathcal H}^{N-1}\\\\
&& = \liminf_{j}\int_{Q\cap S(\tilde{v}_j)} g_1^1([\tilde{v}_j](y),  \nu(\tilde{v}_j(y))) \de {\mathcal H}^{N-1}.\\\\
\end{eqnarray*}

This last equality comes from the fact that $\tilde{v}_j$ has no jumps in $\{ y \in Q: |\chi_{j}(y) - \chi(x_0)| = 1\}$ and
\begin{equation}\label{556} \lim_j \int_{Q\cap S(\tilde{v}_j)\cap S(\tilde{\chi}_j)} g_1^1 ([\tilde{v}_j],  \nu(\tilde{v}_j))\, \de \cH^{N-1} =0.\end{equation}
In fact, by Lemma \ref{385} and $(H_2)$, we have:
\begin{eqnarray*} \lim_j \int_{Q \cap S(\tilde{v}_j)\cap S(\chi_j)} g_1^1 ([\tilde{v}_j], \nu(\tilde{v}_j))\, \de \cH^{N-1} \leq C \lim_j \cH^{N-1}(S(\tilde{v}_j)\cap S(\chi_j))
\leq C \lim_j \cH^{N-1}(S(\chi_j)) \to 0,
\end{eqnarray*}
since $x_0 \notin S(\chi),$ by the appropriate choice of the sizes of the cubes $\delta_j$ so that \eqref{465} holds.

Thus
 \[
 \frac{d\mu}{d\mathcal{L}^N}(x_0)\geq 
 \lim_{j}\left\{\int_Q f(\chi(x_0),\nabla \tilde{v}_j(y)) \de y +
\int_{Q \cap S( v_{j})}g_1^1 ([\tilde{v}_j], \nu_j(y)) \de {\mathcal H}^{N-1}\right\}.
 \]
Since $\tilde{v}_j$ is admissible for the definition of $\tilde{H}$, the proof is concluded.

\subsubsection{Interfacial}

We want to show that 
$$
\frac{d \mathcal{F}_{\text{od-sd}}( \chi, u, G)}{d \cH^{N-1}\lfloor S(\chi, u)}(x_0) \geq \gamma(\chi^+(x_0), \chi^-(x_0), u^+(x_0), u^-(x_0), \nu(\chi, u)(x_0)),
$$ 
namely taking into account Remark \ref{remgammaodsd}, 
\begin{equation}\label{673}
\frac{d\mathcal{F}_{\odsd}( \chi, u, G)}{d \cH^{N-1}\lfloor S(u)}(x_0) \geq \gamma_{\text{sd}}(\chi(x_0), [u](x_0), \nu(u)(x_0)),
\end{equation}
for $\cH^{N-1}$-a.e. $x_0 \in S(u)\setminus S(\chi)$,
\begin{equation}\label{1268}
\frac{d\mathcal{F}_{\odsd}( \chi, u, G)}{d \cH^{N-1}\lfloor S(\chi, u)}(x_0) \geq \gamma( \chi(x_0), [u](x_0), \nu(\chi,u)(x_0)),
\end{equation}
for $\cH^{N-1}$-a.e. $x_0 \in S(u)\cap S(\chi)$,
and
\begin{equation}
\label{1269}
\frac{d\mathcal{F}_{\odsd}( \chi, u, G)}{d \cH^{N-1}\lfloor S(\chi)}(x_0) \geq  |D \chi|(x_0),
\end{equation}
for $\cH^{N-1}$-a.e. $x_0 \in S(\chi)\setminus S(u)$.

Let $U \subset \cO(\Omega)$, open and let $(\chi_n, u_n)$ be an admissible sequence for the definition of $\mathcal{F}_{\odsd}( \chi, u, G)(U)$, such that, for $\eta > 0$ fixed,
\begin{equation}\label{measure}
\begin{split} \eta + \mathcal{F}_{\odsd}( \chi, u, G)(U) \geq& \lim_n \left \{\int_U f(\chi_n, \nabla u_n)\; dx \right.\\
& \left. + \int_{U \cap S(\chi_n, u_n) }g (\chi_n^+, \chi_n^-,u_n^+, u_n^-,  \nu(\chi_n, u_n))\, \de \cH^{N-1} + |D\chi_n|(U) \right\}\\
=& \lim_n \mu_n(U),
\end{split}
\end{equation}
where $\mu_n$ is a bounded sequence of Radon measures, such that, upon a choice of a (non-relabelled) subsequence, $\mu_n \wsto \mu$.
We divide the proof in three parts according to the choice of $x_0.$
First consider $x_0 \in U \cap \left(S(u)\setminus S(\chi)\right)$. In this case, we prove \eqref{673}, taking into account the sequential characterization of $\gamma_{\tiny{\text{sd}}}$ (see Remark \ref{918a}).

The desired lower bound follows from proving that 
\begin{equation*} \frac{d\mu}{d\cH^{N-1}}(x_0) \geq \gamma_{\sd}( \chi(x_0),[u](x_0), \nu(u)(x_0)),\end{equation*}
for $\cH^{N-1}$-a.e. $x_0 \in S(u)\setminus S(\chi)$ and by letting $\eta \to 0$.
Choose a sequence of radii $\delta_k \to 0$ such that $\mu(\partial Q_\nu(x_0, \delta_k)) = 0, \; \forall k \in \N{}$.
Then we have that
\begin{eqnarray*} \frac{d\mu}{d\cH^{N-1}}(x_0) & \geq & \lim_{k,n}\frac{1}{\delta_k^{N-1}} \int_{Q_\nu(x_0, \delta_k) \cap S(u_n)\setminus S(\chi_n)} g_1 ( \chi_n,[u_n], \nu(u_n))\, \de \cH^{N-1}\\\\
&& = \lim_{k,n} \int_{Q_\nu \cap \frac{S(u_n)\setminus S(\chi_n) - x_0}{\delta_k}} g_1( \chi_n(x_0 + \delta_ky),[u_n](x_0 + \delta_ky), \nu(u_n(x_0+\delta_k y))\, \de\cH^{N-1},
\end{eqnarray*}
where $\nu := \nu(u)(x_0).$
Define
\begin{equation}
\label{1664}
\chi_{n,k}(y) := \chi_n(x_0 + \delta_ky),
\end{equation} and
\begin{equation}
\label{1668}
u_{n,k}(y) := u_n(x_0+ \delta_ky) - u^-(x_0).
\end{equation}

The point $x_0 \in S(u)\setminus S(\chi)$ is to be chosen $\cH^{N-1}$-a.e so that
\begin{equation}\label{710} \lim_{k,n} \|\chi_{n,k}(y) - \chi(x_0)\|_{L^1(Q_\nu)} = 0,\end{equation}
 and
\begin{equation}\label{708} \lim_{k,n} \|u_{k,n}(y) - u_{\lambda, \nu}\|_{L^1(Q_\nu; \R d)} = 0,\end{equation}
where $\lambda= |[u]|(x_0)$, $\nu$ is as defined above, and $u_{\lambda,\nu}$ is defined according to \eqref{vlambdanu}.
Following arguments in \cite{CF}, we get a diagonalizing sequence $v_k := u_{n(k),k}$ such that
$$ v_k \to u_{\lambda, \nu} \; \text{in}\, L^1 \; \; \text{and}\; \; \nabla v_k \wto 0 \text{ in } L^p.$$
Let $\tilde{\chi}_k : = \chi_{n(k), k}$. Next we change slightly this sequence in order to fix $\chi(x_0)$. For that, we need to prove that 

$$\lim_{k\to\infty} \int_{\{y \in Q_\nu: |\tilde{\chi}_k(y) - \chi(x_0)| \neq 0\}\cap S(v_k)\setminus S(\tilde{\chi}_k) \}} g_1(\tilde{\chi}_k(y), [v_k](y),  \nu(v_k))\, \de \cH^{N-1}=0,$$
where $g_1(i,\cdot,\cdot)=g_1^i(\cdot,\cdot)$ is as in \eqref{g}. By $(H_2)$, it is enough to prove that
$$\lim_{k\to\infty} \int_{\{y \in Q_\nu: |\tilde{\chi}_k(y) - \chi(x_0)| \neq 0\}\cap S(v_k)\setminus S(\tilde{\chi}_k) \}} | [v_k] |(y)\, \de \cH^{N-1}.$$

Similarly to the lower bound bulk, this is achieved by changing $v_k$.

Define 
$$ \tilde{v}_k := \left\{ \begin{array}{ll}
v_k & \text{in}\; Q_\nu \cap \{|\tilde{\chi}_k(y) - \chi(x_0)| = 0\},\\
u_{\lambda,\nu}& \text{otherwise}\\
\end{array}
\right.
$$
Denoting $F_k :=  \{|\tilde{\chi}_k(y) - \chi(x_0)| = 0\}$, we have that (see the proof of lower bound inequality for bulk) $\mathcal{L}^N (Q_\nu\setminus F_k) \to 0$ and so, we still have that $\tilde{v}_k \to u_{\lambda, \nu}$ in $L^1$. Moreover we clearly still have that, $\nabla \tilde{v}_k \wto 0$ in $L^p$.

Then, taking into account the definition of $\tilde{v}_k$, we have that
\begin{equation*}
\begin{split} 
\frac{d\mu}{d\mathcal {H}^{N-1}}(x_0) \geq & \lim_{k\to\infty} \int_{F_k \cap S(v_k)\setminus S(\tilde{\chi}_k)} g_1 (\chi(x_0),[v_k],  \nu(v_k))\, \de \cH^{N-1}\\
\geq &\lim_{k\to\infty} \int_{F_k \cap S(\tilde{v}_k)\setminus S(\tilde{\chi}_k)} g_1 ( \chi(x_0),[\tilde{v}_k], \nu(\tilde{v}_k))\, \de \cH^{N-1}\\
= &\lim_{k\to\infty} \int_{Q_\nu \cap S(\tilde{v}_k)\setminus S(\tilde{\chi}_k)} g_1 ( \chi(x_0), [\tilde{v}_k],\nu(\tilde{v}_k))\, \de \cH^{N-1} - \lim_{k\to\infty} \int_{(Q_\nu\setminus F_k) \cap S(\tilde{v}_k)\setminus S(\tilde{\chi}_k)} g_1 ( \chi(x_0),[\tilde{v}_k], \nu(\tilde{v}_k))\, \de \cH^{N-1}\\\\
\geq & \lim_{k\to\infty} \int_{Q_\nu \cap S(\tilde{v}_k)\setminus S(\tilde{\chi}_k)} g_1 ( \chi(x_0),[\tilde{v}_k], \nu(\tilde{v}_k))\, \de \cH^{N-1} -\lim_{k\to\infty} \int_{Q_\nu \cap \{ y\cdot\nu = 0\} \cap \{ |\tilde{\chi}_k(y) - \chi(x_0)| = 1\}}\lambda \, \de \cH^{N-1}\\\\
=& T_1 - T_2.
\end{split}
\end{equation*}

For the sake of illustration, we control separately the terms $T_1$ and $T_2$.
Control of $T_2$: setting $y = (\tilde{y}, y_\nu)$
$$ T_2 \leq \int_{Q_\nu \cap  \{ y\cdot\nu = 0\}} |\tilde{\chi}_k(y) - \chi(x_0)| \, \de y_\nu = \lambda \int \int |\tilde{\chi}_k(y) - \chi(x_0)| \, \de y_\nu \de \tilde{y} = \lambda \int_{Q_\nu} |\tilde{\chi}_k(y) - \chi(x_0)| \, \de y \to 0,$$ by \eqref{710}.
It remains to notice that, for what concerns $T_1$, similarly to \eqref{556}, we have that
\begin{equation*}
\lim_{k\to\infty} \int_{Q_\nu \cap S(\tilde{v}_k)\setminus S(\tilde{\chi}_k)} g_1 ([\tilde{v}_k], \chi(x_0), \nu(\tilde{v}_k))\, \de \cH^{N-1} = \lim_{k\to\infty} \int_{Q_\nu \cap S(\tilde{v}_k)} g_1 ([\tilde{v}_k], \chi(x_0), \nu(\tilde{v}_k))\, \de \cH^{N-1},
\end{equation*}
and we conclude that
\begin{equation*}
\begin{split}
\frac{d\mu}{d\mathcal {H}^{N-1}}(x_0)  \geq &\lim_{k\to\infty} \int_{Q_\nu \cap S(\tilde{v}_k)\setminus S(\tilde{\chi}_k)} g_1 ([\tilde{v}_k], \chi(x_0), \nu(\tilde{v}_k))\, \de \cH^{N-1}\\
=& \lim_{k\to\infty} \int_{Q_\nu \cap S(\tilde{v}_k)} g_1 ([\tilde{v}_k], \chi(x_0), \nu(\tilde{v}_k))\, \de \cH^{N-1}
\end{split}
\end{equation*}
Since $\tilde{v}_k$ is admissible for the definition of $\tilde{\gamma}_{\tiny{\text{sd}}}$, \eqref{673} is proved.
We proceed now with the proof of \eqref{1268}

Similarly to the proof of \eqref{673}, we start with an admissible sequence $(\chi_n,u_n)$ in the conditions of  \eqref{measure}.

The desired lower bound follows from proving that
\begin{equation*}
\frac{d\mu}{d\cH^{N-1}}(x_0) \geq \gamma( \chi^+(x_0),\chi^-(x_0), u^+(x_0), u^-(x_0), \nu(\chi, u)(x_0))
\end{equation*}
for $\cH^{N-1}$-a.e. $x_0 \in S(\chi, u).$
Let $\nu := \nu(\chi, u)(x_0)$ and choose a sequence of radii $\delta_k \to 0$ such that $\mu(\partial Q_\nu(x_0, \delta_k)) = 0, \; \forall k \in \N{}$.
Then we have that
\begin{eqnarray*} \frac{d\mu}{d\cH^{N-1}}(x_0) & \geq & \lim_{k,n}\frac{1}{\delta_k^{N-1}} \int_{Q_\nu(x_0, \delta_k) \cap (S(\chi_n, u_n))} g (u_n^+, u_n^-, \chi_n^+, \chi_n^-, \nu(\chi_n, u_n))\, \de \cH^{N-1} + |D\chi_n|(Q_\nu(x_0, \delta_k)\\\\
&& \geq \lim_{k,n} \frac{1}{\delta_k^{N-1}} \int_{Q_\nu(x_0, \delta_k) \cap (S( \chi_n, u_n))} g (u_n^+, u_n^-, \chi_n^+, \chi_n^-, \nu(\chi_n, u_n))\, \de \cH^{N-1} + |D\chi|(Q_\nu).\\
\end{eqnarray*}
Define $\chi_{n,k}$ and $u_{n,k}$ as in \eqref{1664} and \eqref{1668}.

The point $x_0 \in S(u)\cap S(\chi)$ is to be chosen $\cH^{N-1}$-a.e. so that
\eqref{708} and 
\begin{equation*}
\lim_{k,n} \|\chi_{k,n}(y) - \chi_{a, b, \nu}\|_{L^1(Q_\nu)} = 0,
\end{equation*}
hold, with $a = \chi^+(x_0)$ and $b = \chi^-(x_0).$


We have that
\begin{equation*}
\begin{split}
\frac{d\mu}{d\cH^{N-1}}(x_0)& \geq \lim_{k,n}\frac{1}{\delta_k^{N-1}} \int_{Q_\nu(x_0, \delta_k) \cap (S( \chi_n, u_n))} g (u_n^+, u_n^-, \chi_n^+, \chi_n^-, \nu(\chi_n, u_n))\, \de \cH^{N-1} + |D\chi|(Q_\nu)\\
& = \lim_{k,n} \int_{Q_\nu \cap \{ x_0 + \delta_ky\} \in S(u_n)\cap\{ \chi_n (x_0 + \delta_k y )= 1\}} g_1^1([u_n], \nu(u_n))\, \de \cH^{N-1}\\
& + \lim_{k,n} \int_{Q_\nu \cap \{ x_0 + \delta_ky\} \in S(u_n)\cap\{ \chi_n (x_0 + \delta_k y )=0 \}} g_1^0([u_n], \nu(u_n))\, \de \cH^{N-1}\\
& + \lim_{k,n} \int_{Q_\nu \cap \{ x_0 + \delta_ky\} \in S(u_n)\setminus S(\chi_n)}g_2(u_n^+, u_n^-, \chi_n^+, \chi_n^-, \nu(u_n)), \de \cH^{N-1}\\
& + |D\chi|(Q_\nu).
\end{split}
\end{equation*}
The result follows along the lines of what was done for $\gamma_{sd}$, upon finding diagonalizing sequences $v_k$ and $\chi_k$ admissible for $\tilde{\gamma}( \chi^+(x_0),\chi^-(x_0), u^+(x_0), u^-(x_0), \nu(\chi, u)(x_0))$ and relying on Lemma \ref{680bis}.

Finally, for $x_0 \in S(u)\cap S(\chi)$, the proof of \eqref{1269} is an immediate consequence of the lower semicontinuity of $|D \chi|$ and of the following trivial chain of inequalities which holds for every $U \in \cO(\Omega)$:
\begin{equation*}
\begin{split} 
\mathcal{F}_{\text{od-sd}}( \chi, u, G)(U) \geq& \lim_n \bigg \{\int_U f(\chi_n, \nabla u_n)\; dx + \int_{U \cap S(\chi_n, u_n) }g (\chi_n^+, \chi_n^-,u_n^+, u_n^-,  \nu_n)\, \de \cH^{N-1} + |D\chi_n|(U) \bigg\}\\
\geq &\lim_n |D \chi_n|(U).
\end{split}
\end{equation*}


\subsection{Upper bound}\label{ub}

\subsubsection{Bulk}

We assume first that $u \in L^\infty(\Omega;\R d)$.

For $\mathcal{L}^N$-a.e. $x_0 \in \Omega$ we have that:

$$ \frac{d \cF_{\text{od-sd}}( \chi, u, G)}{d \mathcal{L}^N}(x_0) = \lim_{\delta \to 0} \frac{1}{\delta^N} { \cF}_{\text{od-sd}}( \chi, u, G, Q(x_0, \delta)).$$

The point $x_0$ is taken  such that $\frac{d\mathcal{F}_{\text{od-sd}}}{d \mathcal{L}^N} (x_0)$  exists and 
\eqref{440quater}, \eqref{440bis}, \eqref{440ter}, and \eqref{364} hold.


In what follows assume that $\chi(x_0) = 1$, the case $\chi(x_0)=0$ is handled similarly. 
Let $\rho > 0$ small enough and let $w \in SBV(Q; \R{d})\cap L^\infty(Q;\R{d}), \; w\lfloor_{\partial Q} = \nabla u(x_0)x, \; \int_Q \nabla w \, \de x = G(x_0), \; |\nabla w| \in L^p(Q)$ such that

\begin{equation}\label{357}
H(1, u(x_0), G(x_0)) + \rho \geq \int_Q W_1( \nabla w) \, \de x + \int_{Q\cap S(w) } g^1_1( [w], \nu(w))\, \de \cH^{N-1}.
\end{equation}
Note that, due to Lemma \ref{385}, which obviously has an equivalent form in terms of functions defined in the unit cube, we can take 
$w \in SBV(Q; \R{d})\cap L^\infty(Q;\R{d})$ in \eqref{357}.

We construct now admissible sequences $(\chi_{n,\delta},  u_{n, \delta})$ for $\mathcal{F}_{\text{od-sd}}( \chi, u, G, Q(x_0, \delta))$. We take $\chi_{n,\delta} \equiv \chi$ and rely in  \eqref{357} to define $u_{n, \delta}$.

Let $\zeta (x) = w(x) -\nabla u(x_0)x$. Then $\zeta\lfloor_{\partial Q} = 0.$ Extend $\zeta$ by periodicity to all of $\R{N}$.

For each $\delta$, let $\eta_\delta\in SBV(Q(x_0,\delta);\R d)$ be given by Theorem \ref{Al} and such that
\begin{equation*}  \nabla \eta_\delta = G(x) - G(x_0) + \nabla u(x_0) - \nabla u(x) \hbox{ for }{\mathcal L}^N \hbox{-a.e. }x \in Q,
\end{equation*}
\begin{equation}
\label{372bis}
|D \eta_\delta|(Q(x_0,\delta)) \leq C(N)\int_{Q(x_0,\delta)} |G(x)-G(x_0)|+ |\nabla u(x)-\nabla u(x_0)|dx, 
\end{equation}
Moreover, by Remark \ref{Alf}
\begin{equation*}
|| \eta_\delta||_{L^1(Q(x_0, \delta); \R{d})} \leq C ||G(x) - G(x_0) + \nabla u(x_0) - \nabla u(x)||_{L^1(Q(x_0, \delta); \R{d})}.\end{equation*}

By Lemma \ref{ctap}, for each $\delta$, let $\eta_{n, \delta}$ piecewise constant and such that $\eta_{n, \delta} \to -\eta_\delta$ in $L^1(Q(x_0,\delta);\R d)$ as $n \to \infty$. Moreover,
\begin{equation} \label{384}
|D\eta_{n, \delta}|(Q(x_0, \delta)) \to |D\eta_\delta|(Q(x_0, \delta)), \; \; \text{as }\; n \to \infty.
\end{equation}

Define:
\begin{equation}\label{382}
u_{n, \delta}(x) = u(x) + \frac{\delta}{n} \zeta \left( \frac{n(x - x_0)}{\delta}\right) + \eta_\delta + \eta_{n, \delta}.\end{equation}
For fixed $\delta$ it is clear that $u_{n, \delta} \to u \; \; \text{in}\; L^1.$ Moreover,
\begin{equation}
\label{undelta}
\nabla u_{n,\delta}(x)= \nabla \zeta\left(\frac{n(x-x_0)}{\delta}\right)+G(x)-G(x_0)+\nabla u(x_0).
\end{equation}

\noindent By Riemann-Lebesgue Lemma, for fixed $\delta$ and as $n \to\infty$
\begin{equation*} \nabla  \zeta\left( \frac{n(x - x_0)}{\delta}\right) \wto G(x_0) - \nabla u(x_0) \hbox{ in }L^p,\end{equation*}
 and hence, by \eqref{undelta} we also have that
$$ \nabla u_{n, \delta} \wto G(x) \hbox{ in }L^p,$$
that is, $u_{n, \delta}$ is admissible for  $\mathcal{F}_{\text{od-sd}}( \chi, u, G, Q(x_0, \delta))$.

Then, 
\begin{eqnarray*} \frac{d\mathcal{F}_{\text{od-sd}}( \chi, u, G)}{d\mathcal{L}^N} (x_0)& \leq& \lim_{\delta,n} \frac{1}{\delta^N} \left\{ \int_{Q(x_0, \delta)} f( \chi,\nabla u_{n, \delta})\, \de x \right. \\\\
&& \left. + \int_{  Q(x_0,\delta) \cap S(\chi, u_{n, \delta})} g(\chi^+, \chi^-, u_{n, \delta}^+, u_{n, \delta}^-,  \nu( \chi_{n, \delta},u_{n, \delta}))\, \de\cH^{N-1} + |D\chi|(Q(x_0, \delta))\right\} \\\\
&=&   \lim_{\delta,n} \frac{1}{\delta^N} \left\{ \int_{Q(x_0, \delta)} f( \chi,\nabla u_{n, \delta})\, \de x + \int_{Q(x_0, \delta) \cap S(u_{n, \delta})\setminus S(\chi)} g_1 (\chi, [u_{n,\delta}], \nu(u_{n, \delta}))\, \de \cH^{N-1}\right.\\\\
&& \left. + \int_{Q(x_0, \delta) \cap S(u_{n, \delta})\cap S(\chi)} g_2(\chi^+, \chi^-,u_{n, \delta}^+, u_{n, \delta}^-,  \nu( \chi_{n, \delta},u_{n, \delta}))\, \de\cH^{N-1} + |D\chi|(Q(x_0, \delta))\right\}\\\\
&=&  \lim_{\delta,n} \frac{1}{\delta^N} \left\{ \int_{Q(x_0, \delta)\cap \{ \chi = 1\}} f( 1,\nabla u_{n, \delta})\, \de x +  \int_{Q(x_0, \delta)\cap \{ \chi = 0\}} f( 0,\nabla u_{n, \delta})\, \de x\right.\\\\
&& +\int_{Q(x_0, \delta)\cap \{\chi = 1 \}\cap S(u_n, \delta)} g_1^1 ([u_{n,\delta}],  \nu(u_{n, \delta}))\, \de \cH^{N-1} \\\\
&& +\int_{Q(x_0, \delta)\cap \{\chi = 0 \}\cap S(u_n, \delta)} g_1^0 ([u_{n,\delta}],  \nu(u_{n, \delta}))\, \de \cH^{N-1}\\\\
&& \left. + \int_{Q(x_0, \delta) \cap S(u_{n, \delta})\cap S(\chi)} g_2(\chi^+, \chi^-,u_{n, \delta}^+, u_{n, \delta}^-,   \nu( \chi_{n, \delta},u_{n, \delta}))\, \de\cH^{N-1} + |D\chi|(Q(x_0, \delta))\right\}.\\\\
 \end{eqnarray*}
The term $ \frac{1}{\delta^N}|D\chi|(Q(x_0, \delta)) \to 0$ by \eqref{440bis}, so we omit it in the following computations.

We address separately each one of the other terms.
\begin{equation}\label{estimatebulk1} 
\begin{split}
\lim_{\delta,n} \frac{1}{\delta^N} & \bigg\{ \int_{Q(x_0, \delta)\cap \{ \chi = 1\}} f( 1,\nabla u_{n, \delta})\, \de x  + \int_{Q(x_0, \delta)\cap \{ \chi = 0\}} f( 0,\nabla u_{n, \delta})\, \de x \bigg\}  \\ 
=& \lim_{\delta,n} \frac{1}{\delta^N}\bigg\{ \int_{Q(x_0, \delta)} f( 1,\nabla u_{n, \delta})\, \de x +\int_{Q(x_0, \delta)\cap \{ \chi = 0\}} f( 0,\nabla u_{n, \delta})\, \de x - f( 1,\nabla u_{n, \delta})\, \de x  \bigg\}\\ 
\leq & \lim_{\delta,n} \frac{1}{\delta^N} \bigg\{\int_{Q(x_0, \delta)} f( 1,\nabla u_{n, \delta})\, \de x  + 2\beta \int_{Q(x_0, \delta)\cap \{\chi = 0\}} ( 1 + |\nabla u_{n, \delta}|^p)\, \de x\bigg\},
\end{split}
\end{equation}
by $(H_1)$.
By \eqref{undelta} and, since by \eqref{440bis} we have that 
\begin{equation*}
\lim_{\delta \to 0} \frac{1}{\delta^N}\cL^{N}\left(Q(x_0, \delta)\cap \{ \chi(x) = 0\}\right) = 0,
\end{equation*}
together with \eqref{440quater}, we can control all terms in $\int_{Q(x_0, \delta)\cap \{\chi = 0\}} ( 1 + |\nabla u_{n, \delta}|^p)\, \de x$ but the one involving $\nabla \zeta\left( \frac{n(x-x_0)}{\delta}\right)$, which we address now.
We have that 
\begin{equation*}
\lim_{\delta,n} \frac{1}{\delta^N} \int_{Q(x_0, \delta)\cap \{\chi = 0\}}  \left|\nabla \zeta\left( \frac{n(x-x_0)}{\delta}\right)\right|^p\, \de x \to 0,
\end{equation*}
since
\begin{equation*}
\begin{split}
 \lim_{ \delta,n} \frac{1}{\delta^N}  \int_{Q(x_0, \delta)\cap \{\chi = 0\} } \left|\nabla \zeta\left( \frac{n(x-x_0)}{\delta}\right)\right|^p\, \de x 
=& \lim_{\delta,n}\int_{Q(0,1)\cap \{\chi(x_0+ \delta y)=0\}} |\nabla \zeta(ny)|^p \de y  \\
=&\lim_{\delta \to 0}  \int_{\{\chi(x_0+\delta y)=0\}}\|\nabla \zeta\|^p_{L^p(Q(0,1))}=0, 
\end{split}
\end{equation*}
by the Riemann-Lebesgue Lemma and since $\chi(x_0)=1$ and \eqref{440quater} holds.
Thus we can conclude that 
\begin{equation}
\label{1+undelta}
\lim_{\delta,n}\frac{1}{\delta^N}\int_{Q(x_0, \delta)\cap \{\chi = 0\}} ( 1 + |\nabla u_{n, \delta}|^p)\, \de x=0.
\end{equation}
Hence, by \eqref{undelta}, \eqref{1+undelta}, 
and \eqref{estimatebulk1}, it remains to check the following estimate.
Using $(H_1)$, we have
\begin{equation*}
\begin{split}
&\lim_{\delta,n} \frac{1}{\delta^N}  \int_{Q(x_0, \delta)}f( 1,\nabla u_{n, \delta})\, \de x \\
\leq &\lim_{\delta, n}\frac{1}{\delta^N} \left\{ \int_{Q(x_0, \delta)\cap \{ \chi = 1\}} f( 1, \left ( \nabla u(x_0) +  \nabla \zeta\left( \frac{n(x-x_0)}{\delta}\right)\right)\, \de x \right. \\
+ &\left.C\int_{Q(x_0, \delta)} ( |G(x) - G(x_0)|)(|G(x)|^{p-1} + |G(x_0)|^{p-1})+1)\, \de x\right\}.
\end{split}
\end{equation*}
The last term is controlled by \eqref{364}; regarding the remaining term, upon a change of variables and taking into account the periodicity of $\zeta$, we have that
\begin{equation*}
\begin{split}
\lim_{\delta,n}  \frac{1}{\delta^N} & \int_{Q(x_0, \delta)} f \bigg( 1, \nabla u(x_0) + \nabla \zeta \Big( \frac{n(x - x_0)}{\delta}\Big)\bigg)\, \de x = \lim_{\delta, n}\frac{1}{n^N} \int_{nQ}f(1(,\nabla u(x_0) + \nabla \zeta(y))\, \de y \\
 & = \int_Q f(1, \nabla u(x_0) + \nabla \zeta(y))\, \de y = \int_Q f(1, \nabla w(y))\, \de y.\\
\end{split}
\end{equation*}
By $(H_2)$ we have
\begin{equation}\label{852}
\begin{split}
&\lim_{\delta,n} \frac{1}{\delta^N} \left\{ \int_{ Q(x_0, \delta)\cap \{\chi = 1 \}\cap S(u_{n, \delta})} g_1^1 ([u_{n,\delta}], \nu(u_{n, \delta}))\, \de \cH^{N-1}\right. \\
&+\left.\int_{Q(x_0, \delta)\cap \{\chi = 0 \}\cap S(u_{n, \delta})} g_1^0 ([u_{n,\delta}], \nu(u_{n, \delta}))\, \de \cH^{N-1}\right\}\\
& = \lim_{\delta,n} \frac{1}{\delta^N} \left\{ \int_{Q(x_0, \delta) \cap (S(u_{n, \delta})\setminus S(\chi))} g_1^1 ([u_{n,\delta}],  \nu(u_{n, \delta}))\, \de \cH^{N-1}\right.\\
& + \left. \int_{Q(x_0, \delta)\cap \{\chi = 0 \}\cap (S(u_{n, \delta})\setminus S(\chi))} g_1^0 ([u_{n,\delta}], \nu(u_{n, \delta})) -  g_1^1 ([u_{n,\delta}], \nu(u_{n, \delta}))\, \de \cH^{N-1}\right\}\\
& \leq  \lim_{\delta,n} \frac{1}{\delta^N} \left\{\int_{ Q(x_0, \delta) \cap S(u_{n, \delta})\setminus S(\chi)} g_1^1 ([u_{n,\delta}], \nu(u_{n, \delta}))\, \de \cH^{N-1}\right.\\
& + \left. C\int_{Q(x_0, \delta)\cap \{\chi = 0 \}\cap S(u_{n, \delta})} |[u_{n, \delta}]| \, \de \cH^{N-1} \right\},
\end{split}
\end{equation}
and by \eqref{382},
$$ |[u_{n, \delta}]|(x) \leq |[u]|(x) + \frac{\delta}{n}|[\zeta]|\left(\frac{n(x-x_0)}{\delta}\right) + |[\eta_\delta]|(x) + |[\eta_{n,\delta}]|(x).$$
The term
\begin{equation}\label{858}\lim_{\delta, n}\frac{1}{\delta^N} \int_{Q(x_0, \delta)\cap \{\chi = 0 \}\cap S(u)} |[u]| \, \de \cH^{N-1}\end{equation}
is controlled by \eqref{440bis}, while the term
\begin{equation}\label{860}\lim_{\delta, n}\frac{1}{\delta^N} \int_{ Q(x_0, \delta)\cap \{\chi = 0 \}\cap S(\eta_\delta)} |[\eta_\delta]| \, \de \cH^{N-1}\end{equation}
is controlled by \eqref{372bis} and \eqref{364}. 
The control of the term
\begin{equation}\label{863}\lim_{\delta, n}\frac{1}{\delta^N} \int_{Q(x_0, \delta)\cap \{\chi = 0 \}\cap S(\eta_{n, \delta}} |[\eta_{n,\delta}]| \, \de \cH^{N-1}\end{equation}
is similar, taking into account \eqref{384}.
Finally, we have
\begin{equation}\label{881}
\begin{split}
&\lim_{\delta, n}\frac{1}{\delta^N} \int_{Q(x_0, \delta)\cap \{\chi = 0 \}\cap \{ \frac{n(x-x_0)}{\delta} \in S(\zeta)\}} \frac{\delta}{n}|[\zeta]|\left(\frac{n(x-x_0)}{\delta}\right) \, \de \cH^{N-1}(x)\\
&=\lim_{\delta, n}\frac{1}{\delta^N} \int_{Q(x_0, \delta)\cap \{\chi = 0 \}\cap \{ \frac{n(x-x_0)}{\delta }\in S(w)\}}\frac{\delta}{n} |[w]|\left(\frac{n(x-x_0)}{\delta}\right) \, \de \cH^{N-1}(x)\\
& = \lim_{\delta, n}\frac{1}{\delta^N}\frac{\delta}{n}\frac{\delta^{N-1}}{n^{N-1}}\int_{nQ  \cap \{\chi(x_0 + \frac{\delta}{n}y) = 0\}\cap S(w)} |[w]|\, \de \cH^{N-1}(y)\\
& =  \lim_{\delta, n}\frac{1}{\delta^N}\frac{\delta}{n}\frac{\delta^{N-1}}{n^{N-1}}n^N \int_{Q\cap \{\chi(x_0 + \frac{\delta}{n}y) = 0\}\cap S(w)} |[w]|(y)\, \de \cH^{N-1}(y)\\
& = \lim_{\delta, n} \int_{Q\cap S(w)} |\chi(x_0 + \frac{\delta}{n}y) - \chi(x_0)| |[w]|(y)\, \de \cH^{N-1}(y) = 0,\\
\end{split}
\end{equation}
 since $|\chi(x_0 + \delta y) - \chi(x_0)| \to 0$ for $\cH^{N-1}$-a.e. $y \in S(w)$ (see \cite[Theorem 3.108]{AFP}).

Therefore, going back to equation \eqref{852}, we have that
\begin{equation*}
\begin{split}
&\lim_{\delta,n} \frac{1}{\delta^N} \left\{ \int_{Q(x_0, \delta)\cap \{\chi = 1 \}\cap S(u_{n, \delta})} g_1^1 ([u_{n,\delta}], \nu(u_{n, \delta}))\, \de \cH^{N-1}\right. \\
&+\left.\int_{Q(x_0, \delta)\cap \{\chi = 0 \}\cap S(u_{n, \delta})} g_1^0 ([u_{n,\delta}],  \nu(u_{n, \delta}))\, \de \cH^{N-1}\right\}\\
& = \lim_{\delta,n} \frac{1}{\delta^N}\int_{Q(x_0, \delta) \cap (S(u_{n, \delta})\setminus S(\chi))} g_1^1 ([u_{n,\delta}],  \nu(u_{n, \delta}))\, \de \cH^{N-1} \\
& \leq \int_{Q(x_0, \delta) \cap (S(w)\setminus S(\chi))} g_1^1 ([w], \nu(w))\, \de \cH^{N-1}, \\
\end{split}
\end{equation*}
where the last equality follows from arguments similar to the ones used in \eqref{858}, \eqref{860}, \eqref{863}, \eqref{881}, and since $w \in L^\infty$.

We still have to show that
$$\lim_{\delta, n} \frac{1}{\delta^N}\int_{Q(x_0, \delta) \cap (S(u_{n, \delta}) \cap S(\chi))} g_2 (\chi^+, \chi^-,u_{n, \delta}^+, u_{n, \delta}^- , \nu(u_{n, \delta}))\, \de \cH^{N-1} = 0.$$
By $(H_5)$, \eqref{382}, the fact that $ \frac{1}{\delta^N}|D\chi|(Q(x_0, \delta)) \to 0$, and by arguments that were used before, we just have to show that

$$ \lim_{\delta, n} \frac{1}{\delta^N} \int_{  Q(x_0, \delta) \cap \{ x_0 + \frac{\delta}{n}S(w)\} \cap S(\chi)}  |[w]|\left( \frac{n(x - x_0)}{\delta}\right) \, \de \cH^{N-1} = 0.$$

Similarly to the previous calculations

\begin{eqnarray*} && \frac{1}{\delta^N} \int_{Q(x_0, \delta) \cap  \{ x_0 + \frac{\delta}{n}S(w)\} \cap S(\chi) }  \frac{\delta}{n} |[w]|\left(\frac{n(x-x_0)}{\delta}\right)\de \cH^{N-1}\\\\
&& = \frac{1}{\delta^N}\frac{\delta}{n} \frac{\delta^{N-1}}{n^{N-1}} \int_{nQ \cap \{ x_0 + \frac{\delta}{n}y \in S(\chi)\}\cap S(w)} | [w(y)]| )\, \de \cH^{N-1} = \frac{1}{n^N}{n^N} \int_{Q \cap \{ x_0 + \frac{\delta}{n}y \in S(\chi)\}\cap S(w)} |[w(y)]|\, \de \cH^{N-1}\\\\
&& \leq C\cH^{N-1}(\{x_0 + \frac{\delta}{n}y \in S(\chi)\}\cap Q(x_0, \delta))\to 0.\\
\end{eqnarray*}
The desired upper bound  follows from \eqref{357} by letting $\rho \to 0$.

\subsubsection{ Interfacial }

We want to show that for $\cH^{N-1}$-a.e. $x_0 \in S(\chi, u)$

\begin{equation}\nonumber
\frac{d \mathcal{F}_{\text{od-sd}}( \chi, u, G)}{d \cH^{N-1}\lfloor S(\chi, u)}(x_0) \leq \gamma(\chi^+, \chi^-, u^+, u^+, \nu(\chi, u))(x_0),
\end{equation}
which by virtue of Remark \ref{remgammaodsd} can be decomposed as follows:

- for $\cH^{N-1}$-a.e. $x_0 \in S(u)\setminus S(\chi)$
\begin{equation}\label{445}
\frac{d \mathcal{F}_{\text{od-sd}}( \chi, u, G)}{d \cH^{N-1}\lfloor S( u)}(x_0) \leq \gamma_{\text{sd}}( [u], \chi,\nu(u))(x_0),
\end{equation}

- for $\cH^{N-1}$-a.e. $x_0 \in S(u)\cap S(\chi)$

\begin{equation}\label{879}
\frac{d \mathcal{F}_{\text{od-sd}}( \chi, u, G)}{d \cH^{N-1}\lfloor S(\chi, u)}(x_0) \leq \gamma(\chi^+, \chi^-, u^+, u^+, \nu(\chi, u))(x_0),
\end{equation}
and
\\
\vspace{0.2cm}

- for $\cH^{N-1}$-a.e. $x_0 \in S(\chi)\setminus S(u)$
\begin{equation}\label{879chi}
\frac{d \mathcal{F}_{\text{od-sd}}( \chi, u, G)}{d \cH^{N-1}\lfloor S(\chi)}(x_0) \leq   |D \chi|(x_0).
\end{equation}

Following an argument of \cite[Proposition 4.8]{AMT}, in view of the continuity properties of $\gamma$ proven in Lemma \ref{Lipgammaodsd} it suffices to consider the couple $(\chi,u)= (a\chi_E+b \chi_{\Omega \setminus E},c\chi_E+d\chi_{\Omega \setminus E})$, with $a, b \in \{0,1\}$, $c,d \in \R {d}$, and  $\chi_E$ the characteristic function of a set $E$ of finite perimeter.  We consider first the case where the set $E$ is a polyhedron and then any set of finite perimeter.\\

\paragraph{\textbf{$E$ polyhedral set}} 
Covering $\Omega$ via Besicovitch Theorem 
with disjoint open cubes $Q_{\nu(x_i)}(x_i, \varepsilon_i)$, centered at points of approximate continuity for $\gamma(a,b,c,d,\nu(x))$ with respect to ${\mathcal H}^{N-1}\lfloor{S(\chi,u)}$, one can restrict the analysis to a single cube.  
Indeed, it is enough to prove the upper bound inequality for the case $\Omega = Q_\nu$, with $\nu = e_N$ and for $ u = u_{c, d, \nu}(x_0), c= u^+(x_0), d = u^-(x_0)$ and $\chi = \chi_{a,b, \nu}(x_0), $ with $a= \chi^+(x_0)$ and $b = \chi^-(x_0).$

 
  
 
We start with the proof of \eqref{879}.
For $\rho > 0$ let $w \in SBV(Q; \R{d}), w|_{\partial Q}=  u_{c, d, \nu}$ satisfying $\int_Q \nabla w = 0$ and $\tilde{\chi} \in BV(Q; \{0,1\}), \; \tilde{\chi}|_{\partial Q} = \chi_{a, b, \nu}$ such that
\begin{equation*}
\begin{split}
\gamma(\chi^+, \chi^-,u^+, u^-, \nu(\chi, u))(x_0) + \rho &\geq \int_{Q \cap S(w)\cap S(\tilde{\chi})} g_2(\tilde{\chi}^+, \tilde{\chi}^-, w^+, w^-, \nu(\tilde{\chi}, w))\, \de \cH^{N-1}\\
 &+ \int_{Q\cap \{\tilde{\chi} = 1\} \cap S(w)} g_1^1( [w], \nu(w))\, \de \cH^{N-1} \\
 & + \int_{Q\cap \{\tilde{\chi} = 0\} \cap S(w)} g_1^0( [w], \nu (w))\, \de \cH^{N-1}\\
 & + |D\tilde{\chi}|(Q).
 \end{split}
\end{equation*}
For $\delta>0$ small enough, and $n\in\N{}$, define
\begin{eqnarray}
D_n(x_0,\delta) & := & Q(x_0,\delta)\cap\left\{x:\frac{|(x-x_0)\cdot e_N|}\delta<\frac1{2n}\right\}, \nonumber\\
Q^+(x_0,\delta) & := & Q(x_0,\delta)\cap\left\{x:\frac{(x-x_0)\cdot e_N}\delta>0\right\}, \nonumber\\
Q^-(x_0,\delta) & := & Q(x_0,\delta)\cap\left\{x:\frac{(x-x_0)\cdot e_N}\delta<0\right\}.\nonumber
\end{eqnarray}
Extend $w(\cdot, y_N)$ by $Q'$-periodicity ($Q' := \{ y \in Q : y_N = 0\}$) and construct the sequence
\begin{equation}\nonumber
w_{n,\delta}(x):=\begin{cases}
u^+(x_0) & x\in Q^+(x_0,\delta)\setminus D_n(x_0,\delta), \\
w\left(\frac{n(x-x_0)}{\delta}\right) & x\in D_n(x_0,\delta), \\
u^-(x_0) & x\in Q^-(x_0,\delta)\setminus D_n(x_0,\delta).
\end{cases}
\end{equation}
Notice that, arguing as in \cite[Theorem 4.4 -- Upper bound]{CF},   $w_{n,\delta}\stackrel{L^1}{\to} u_{u^+(x_0), u^-(x_0), \nu}$ and $\nabla w_{n, \delta} \wto 0$ in $L^p$ as $n\to\infty$.

Let now $h_\delta$ given by Theorem \ref{Al} be such that $\nabla h = G$  in $Q(x_0, \delta)$ and satisfying
\begin{equation}\label{507}|Dh_\delta|(Q(x_0, \delta)) \leq C(N) \int_{Q(x_0, \delta)} |G(x)|\; dx, \end{equation}
and let $h_{n,\delta}$ be a sequence of piecewise constant functions given by Lemma \ref{ctap} such that
$h_{n,\delta} \to - h_\delta$ in $L^1$ as $n \to \infty$ and
\begin{equation} \label{510}
|Dh_{n, \delta}|(Q(x_0, \delta)) \to |Dh_\delta|(Q(x_0, \delta)), \; \; \text{as }\; n \to \infty.
\end{equation} Define the sequence
\begin{equation}\label{493} u_{n, \delta} :=  w_{n,\delta} + h_\delta + h_{n, \delta}.\end{equation}

Similarly, extend $\tilde{\chi}$ by periodicity and define:
\begin{equation}\nonumber
\chi_{n,\delta}(x):=\begin{cases}
\chi^+(x_0) & x\in Q^+(x_0,\delta)\setminus D_n(x_0,\delta), \\
\tilde{\chi}\left(\frac{n(x-x_0)}{\delta}\right) & x\in D_n(x_0,\delta), \\
\chi^-(x_0) & x\in Q^-(x_0,\delta)\setminus D_n(x_0,\delta).
\end{cases}
\end{equation}

Clearly, for fixed $\delta$, the sequences $u_{n, \delta} $ is admissible for the definition of $\gamma(u^+, u^-, \chi^+,\chi^-, \nu(\chi, u))(x_0).$ 

Regarding $\chi_{n, \delta}$, we have that $\chi_{n, \delta} \to \chi_{a, b, \nu}(x_0)$ in $L^1$ as $n \to \infty$, and $|D\chi_{n, \delta}|$ is bounded uniformly with respect to $n$. Therefore, by Proposition 3.12 in \cite{AFP} we have that $\chi_{n,\delta} \wsto \chi_{a,b,e_N}$ in BV as $n\to \infty.$

Then we have that $u_{n, \delta} $ and $\chi_{n, \delta}$ are admissible for $\mathcal{F}_{\text{od-sd}}( \chi, u, G)(Q(x_0, \delta)$ and so, 
\begin{eqnarray*}
\frac{d \mathcal{F}_{\odsd}( \chi, u, G)}{d \cH^{N-1}}(x_0) &\leq& \lim_{\delta, n} \frac{1}{\delta^{N-1}} \left\{ \int_{Q(x_0, \delta)} f(\chi_{n, \delta},\nabla u_{n, \delta})\, \de x\right.\\
&& + \int_{Q^+(x_0, \delta) \cap S(\chi_{n, \delta},u_{n, \delta}) }g(\chi_{n, \delta}^+, \chi_{n,\delta}^-,u_{n, \delta}^+,u_{n, \delta}^-, \nu(\chi_{n, \delta},u_{n,\delta}))\, \de \cH^{N-1}\\
&& + |D\chi_{n, \delta}|(Q(x_0, \delta)) \Big\}\\
&& = L_1 + L_2 + L_3.\\
\end{eqnarray*}
As in \cite{CF}, the term $L_1$ is controlled by $(H_1)$ and choosing $x_0$ ($\cH^{N-1}$-a.e. in $S(\chi,u)$) so that
\begin{equation*}
\lim_{\delta \to 0} \frac{1}{\delta^{N-1}} \int_{Q(x_0, \delta)} |G|^p\, \de x = 0.
\end{equation*}
The term $L_3$, upon a change of variables gives trivially $|D\tilde{\chi}|(Q).$
It remains to control $L_2$. We have that

\begin{equation*}
\begin{split}
L_2 =&   \lim_{\delta, n} \frac{1}{\delta^{N-1}}\int_{Q(x_0, \delta) \cap S(u_{n, \delta})\cap S(\chi_{n, \delta})  }g_2(\chi_{n, \delta}^+, \chi_{n,\delta}^-,u_{n, \delta}^+,u_{n, \delta}^-, \nu(\chi_{n, \delta},u_{n,\delta}))\, \de \cH^{N-1}\\
& + \lim_{\delta, n} \frac{1}{\delta^{N-1}}\int_{Q(x_0, \delta) \cap S(u_{n, \delta})\setminus S(\chi_{n, \delta}))  }g_1(\chi_{n, \delta},[u_{n, \delta}],  \nu(u_{n, \delta}))\, \de \cH^{N-1}\\
=& K_1 + K_2.\\
\end{split}
\end{equation*}
The term $K_2$ can be written as:

\begin{equation*}
\begin{split}
K_2 & =  \lim_{\delta, n} \frac{1}{\delta^{N-1}}\int_{Q(x_0, \delta) \cap \{ \chi_{n, \delta} = 1\} \cap (S(u_{n, \delta})\setminus S(\chi_{n, \delta}))  }g_1^1([u_{n, \delta}], \nu(u_{n, \delta}))\, \de \cH^{N-1}\\
& + \lim_{\delta, n} \frac{1}{\delta^{N-1}}\int_{Q(x_0, \delta) \cap \{ \chi_{n, \delta} = 0\} \cap (S(u_{n, \delta})\setminus S(\chi_{n, \delta})) }g_1^0([u_{n, \delta}], \nu(u_{n, \delta}))\, \de \cH^{N-1}\\
& = \lim_{\delta, n} \frac{1}{\delta^{N-1}}\int_{Q(x_0, \delta) \cap \{ \chi_{n, \delta} = 1\}  \cap (S(w_{n, \delta})\setminus S(\chi_{n, \delta})) }g_1^1([w_{n, \delta}], \nu(w_{n, \delta}))\, \de \cH^{N-1}\\
& + \lim_{\delta, n} \frac{1}{\delta^{N-1}}\int_{Q(x_0, \delta) \cap \{ \chi_{n, \delta} = 0\} \cap (S(w_{n, \delta})\setminus S(\chi_{n, \delta}))  }g_1^0([w_{n, \delta}], \nu(w_{n, \delta}))\, \de \cH^{N-1}\\
& = \lim_{\delta, n} \frac{1}{\delta^{N-1}}\int_{D_n(x_0, \delta) \cap \{ x_0 + \frac{\delta}{n} \in \tilde{\chi} = 1\}\cap \{x_0 + \frac{\delta}{n}S(w)\}}g_1^1\left(\left[w\left(\frac{n(x-x_0)}{\delta}\right)\right], \nu(w_{n, \delta})\right)\, \de \cH^{N-1}\\
& + \lim_{\delta, n} \frac{1}{\delta^{N-1}}\int_{D_n(x_0, \delta) \cap \{ x_0 + \frac{\delta}{n} \in \tilde{\chi} = 0\}\cap \{x_0 + \frac{\delta}{n}S(w)\}}g_1^0\left(\left[w\left(\frac{n(x-x_0)}{\delta}\right)\right], \nu(w_{n, \delta})\right)\, \de \cH^{N-1}\\
& = \int_{Q\cap \{ \tilde{\chi} = 1\}\cap S(w)} g_1^1([w], \nu(w))\, \de\cH^{N-1} + \int_{Q\cap \{ \tilde{\chi} = 0\}\cap S(w)} g_1^0([w], \nu(w))\, \de\cH^{N-1}.
\end{split}
\end{equation*}
by the definition of $u_{n, \delta}$ (see \eqref{493}), by $(H_2), (H_4)$ , by \eqref{510} and \eqref{507} and by a change of variables and periodicity of both $w$ and $\tilde{\chi}$.
It remains to control $K_1$. Writing for simplicity $\nu_{n,\delta} := \nu(\chi_{n, \delta},u_{n,\delta})$ in what follows, we have that:
\begin{equation}\nonumber
\begin{split}
K_1 = & \lim_{\delta, n} \frac{1}{\delta^{N-1}}\int _{Q(x_0, \delta)\cap S(u_{n, \delta})\cap S(\chi_{n, \delta})} g_2 (\chi_{n, \delta}^+,\chi_{n, \delta}^-,u_{n, \delta}^+, u_{n, \delta}^-,  \nu_{n,\delta})\, \de \cH^{N-1}\\
& = \lim_{\delta, n} \frac{1}{\delta^{N-1}}\int _{D_n(x_0, \delta)\cap S(w_{n, \delta})\cap S(\chi_{n, \delta})} g_2 (\chi_{n, \delta}^+,\chi_{n, \delta}^-,w_{n, \delta}^+, w_{n, \delta}^-,  \nu_{n,\delta})\, \de \cH^{N-1}\\
& + \lim_{\delta, n} \frac{1}{\delta^{N-1}}\int _{Q(x_0, \delta)\cap S(w_{n, \delta})\cap S(u_{n,\delta})\cap S(\chi_{n, \delta})} (g_2 ( \chi_{n, \delta}^+,\chi_{n, \delta}^-,u_{n, \delta}^+, u_{n, \delta}^-,\nu_{n,\delta})-g_2 (\chi_{n, \delta}^+,\chi_{n, \delta}^-,w_{n, \delta}^+, w_{n, \delta}^-, \nu_{n,\delta}))\, \de \cH^{N-1}\\
& -\lim_{\delta, n} \frac{1}{\delta^{N-1}}\int _{D_n(x_0, \delta)\cap (S(w_{ \delta})\setminus S(u_{n,\delta}))\cap S(\chi_{n, \delta})} g_2 (\chi_{n, \delta}^+,\chi_{n, \delta}^-,w_{n, \delta}^+, w_{n,\delta}^-,  \nu_{n,\delta})\, \de \cH^{N-1}\\
&\leq \lim_{\delta, n} \frac{1}{\delta^{N-1}}\int _{D_n(x_0, \delta)\cap S(w_{n, \delta})\cap S(\chi_{n, \delta})} g_2 (\chi_{n, \delta}^+,\chi_{n, \delta}^-,w_{n, \delta}^+, w_{n, \delta}^-,  \nu_{n,\delta})\, \de \cH^{N-1}\\
& + \lim_{\delta, n} \frac{1}{\delta^{N-1}}\int _{Q(x_0, \delta)\cap (S(h_{n, \delta})\cup S(h_{\delta}))\cap S(\chi_{n, \delta})} C\left |[h_{n,\delta}]|+ |[h_\delta]\right | \de \cH^{N-1},
\end{split}
\end{equation}
where we used $(H_7)$.
We observe that
\begin{equation}\nonumber
\begin{split}
& \lim_{\delta, n} \frac{1}{\delta^{N-1}}\int _{Q(x_0, \delta)\cap S(h_{n, \delta})\cap S(\chi_{n, \delta})} |[h_{n, \delta}]| \, \de \cH^{N-1}\\
&  \leq C\lim_{\delta, n} \frac{1}{\delta^{N-1}}\int _{Q(x_0, \delta)} |G(x)|\,\de x  \leq C\lim_{\delta, n} \frac{1}{\delta^{N}}\int _{Q(x_0, \delta)} |G(x)|\,\de x = 0\\
\end{split}
\end{equation}
by \eqref{507} and the choice of $x_0$.
The control of 
$$\lim_{\delta, n} \frac{1}{\delta^{N-1}}\int _{Q(x_0, \delta)\cap S(h_{ \delta})\cap S(\chi_{n, \delta})} |[h_{ \delta}]| \de \cH^{N-1}$$
follows from the same type of estimates together with \eqref{510}.
The result now is an immediate consequence of periodicity and a change of variables. In fact, we have that
\begin{equation}\nonumber
K_1\leq \lim_{\delta, n} \frac{1}{\delta^{N-1}}\int_{Q\cap S(w)\cap S(\tilde{\chi})} g_2 ( \tilde{\chi}^+, \tilde{\chi}^-, w^+, w^-, \nu(\tilde{\chi},w) )\, \de \cH^{N-1}.
\end{equation}
This concludes the proof of \eqref{879}.
\\
The proof of \eqref{445}, since $x_0 \notin S(\chi)$, is simpler and follows  the previous arguments for the proof of equation \eqref{879}, taking $\chi_{n, \delta}$ to be the constant sequence $\chi_{n, \delta} = \chi_{a, b, e_N}$ with $a=b= \chi(x_0).$ We skip the proof.

Finally in order to prove the upper bound when $x_0\in S(\chi)\setminus S(u)$, i.e. \eqref{879chi} it suffices to consider $\chi_{n,\delta}\equiv \chi_{a,b,\nu}$ and $u_{n,\delta}$ constant. 

\paragraph{\textbf{$E$ set of finite perimeter}}.  For every fixed quadruple $(a,b,c,d) \in \{0, 1\}\times \{0, 1\}\times \R {d}\times \R {d}$, in view of the upper semicontinuity of $\gamma$ with respect to the normal variable (see Lemma \ref{Lipgammaodsd}), there exists a sequence $\gamma_m:\R N\to [0,+\infty)$ such that 
 \begin{equation}\nonumber
 \gamma(a,b,c,d,p)\leq \gamma_m(a,b,c,d,p)\leq C|p|, \;\; \hbox{ for every } p \in \R N,\end{equation}
and
\begin{equation}\nonumber
\gamma(a,b,c,d,p)=\inf_m \gamma_m(p),
\end{equation}
where, with an abuse of notations, $\gamma$ has been extended to $\R N$ as a positive $1$-homogeneous function.
 
Consider now a sequence of polyhedra approximating $E$ in the sense of Lemma \ref{polyhedra} and define the sequence  $(\chi_n,u_n) =(a\chi_{E_n}+b\chi_{\Omega\setminus E_n}, c\chi_{E_n}+d\chi_{\Omega \setminus E_n})$.  
From  \eqref{879} we have that, for every $U\in {\mathcal O}(\Omega)$ and  for any $n \in \mathbb N,$
 $$
 {\mathcal F}_{\odsd}(\chi, u,G;U)\leq C(N) \int_U|G(x)|^pdx +\int_{\partial E_n\cap U}\gamma(a,b,c,d,\nu(\chi_n, u_n)(x))d {\mathcal H}^{N-1}.
 $$
 Thus, taking into account the upper bound of $\gamma$ in terms of $\gamma_m$, we obtain
\begin{equation*}
\begin{split}
{\mathcal F}_{\odsd}(\chi, u,G;U) &\leq \liminf_{n\to \infty} {\mathcal F}_{\odsd}(\chi_n, u_n,G;U)\\
&\leq \liminf_{n\to \infty}\left\{C(N) \int_U|G(x)|^pdx +\int_{\partial E_n\cap U}\gamma(a,b,c,d,\nu(\chi_n, u_n)(x))d {\mathcal H}^{N-1}\right\}\\
&\leq \liminf_{n, m\to \infty}\left\{C(N) \int_U|G(x)|^pdx +\int_{\partial E_n\cap U}\gamma_m(a,b,c,d,\nu(\chi_n, u_n)(x))d {\mathcal H}^{N-1}\right\}\\
&\leq \liminf_{n\to \infty} \left\{C(N)\int_U|G(x)|^pdx +\int_{\partial E\cap U}\gamma(a,b,c,d,\nu(\chi_n, u_n)(x))d {\mathcal H}^{N-1}\right\},
\end{split}
\end{equation*}
by first sending $m \to \infty$ and by the Monotone Convergence Theorem.
Finally, by the Dominated Convergence Theorem we get
 $$
 {\mathcal F}_{\odsd}(\chi, u,G; U) \leq  C(N)\int_U|G(x)|^p\de x +\int_{\partial E\cap U}\gamma(a,b,c,d,\nu(x))d {\mathcal H}^{N-1}.
 $$
Taking the Radon-Nykodym derivative at $x_0$ point of absolute continuity for $\gamma(\chi^+(\cdot),\chi^-(\cdot),u^+(\cdot),u^-(\cdot),\nu(\cdot))$ with respect to the ${\mathcal H}^{N-1}$ measure gives the desired result.

\subsection{Completion of the proof of Theorem \ref{312}}
Putting together the results obtained in subsection \ref{lb} and \ref{ub} we have  proven that 
\begin{equation}
\label{replinfty}
\mathcal{F}_{\odsd}( \chi, u, G) = \int_\Omega H ( \chi, \nabla u, G)\, \de x + \int_{\Omega \cap S(\chi, u)} \gamma( \chi^+, \chi^-, u^+, u^-, \nu(\chi, u))\, \de \cH^{N-1}.
\end{equation}
for every $\chi \in BV(\Omega;\{0,1\})$ and $u \in SBV(\Omega;\R d)\cap L^\infty(\Omega;\R d)$.

In order to achieve the representation for every $u \in SBV(\Omega,\R d)$ we start observing that the proof of the lower bound did not exploit the fact that $u\in L^\infty$. Thus it remains to deduce the upper bound in the general case. To this end, define
\begin{equation*}
\begin{split}
\mathcal{J}_{\odsd}( \chi, u, G) &= \int_\Omega H ( \chi, \nabla u, G)\, \de x + \int_{\Omega \cap (S(u)\setminus S(\chi))} \gamma_{\tiny{\text{sd}}}( \chi, [u], \nu)\, \de \cH^{N-1} \\
&+ \int_{\Omega \cap S(u) \cap S(\chi))}\gamma_{\tiny{\text{od-sd}}}(\chi^+, \chi^-,u^+, u^-,  \nu(\chi, u)) \, \de \cH^{N-1}\\
&+ \int_{\Omega \cap (S(\chi) \setminus S(u))}\gamma_{\tiny{\text{od}}}([\chi], \nu)\, \de \cH^{N-1}.
\end{split}
\end{equation*}
Define $\phi_i$ as in \eqref{phii} of Lemma \ref{385}, such that $\phi_i \in C^\infty(\R d,\R d)$ and  $\|\nabla\phi_i\|_{L^\infty}\leq 1$.
Since $\phi_i(u)\to u$ in $L^1$ as $i\to \infty$, the lower semicontinuity of ${\mathcal F}_{\odsd}(\chi,\cdot, G)$ entails that
$$
{\mathcal F}_{\odsd}(\chi,u, G)\leq \liminf_{i\to \infty}{\mathcal F}_{\odsd}(\chi,\phi_i(u), G)=\liminf_{i\to \infty}\mathcal{J}_{\text{od-sd}}(\chi,\phi_i(u), G),
$$
where in the latter equality it has been exploited \eqref{replinfty} and the definition of $\mathcal{J}_{\odsd}$.

We recall that \eqref{422CF}, \eqref{gammasdgrowth}, and \eqref{gammagrowth} hold.
Then, defining for every $i \in \mathbb N$ 
\begin{equation*}
\Omega_i:=\{x\in \Omega: |u^+(x)|\geq e^i \hbox{ or }|u^-(x)|\geq e^i\}\cap\{x \in \Omega: |u^+(x)|<e^{i+1}\hbox{ or } |u^-(x)|<e^{i+1}\},
\end{equation*}
we have
\begin{equation*}
{\mathcal J}_{\odsd}(\chi,\phi_i(u), G) \leq 
\mathcal{J}_{\odsd}(\chi, u ,G ) + C \int_{\{x: |u(x)|\geq e^i\}} (1+|\nabla \phi_i(u)|+ |G|^p)\de x 
+C\int_{\Omega_i\cap S(u)} (1+ |[u]|)(x)|\de {\mathcal H}^{N-1}.
\end{equation*}
Exactly the same arguments in \cite[formula (3.19)-(3.23)]{CF} guarantee that the latter integrals are $O\left(\frac{1}{i}\right)$, hence, letting $i \to \infty$, one can conclude that
$$
{\mathcal F}_{\odsd}(\chi, u,G)\leq \mathcal{J}_{\odsd}(\chi,u,G).
$$

\bigskip

\noindent\textbf{Acknowledgements}.
The authors are grateful to SISSA for its hospitality and support, where part of this research was conducted.
The authors also thank David R.\@ Owen for his fruitful comments on the model.
The research of JM was partially supported by the Funda\c{c}\~{a}o para a Ci\^{e}ncia e a Tecnologia through grant UID/MAT/04459/2013.
The research of MM was partially supported by the European Research Council through the ERC Advanced Grant ``QuaDynEvoPro'', grant agreement no.\@ 290888.
MM is a member of the Progetto di Ricerca GNAMPA-INdAM 2015 ``Fenomeni critici nella meccanica dei materiali: un approccio variazionale''.
MM and EZ are members of the Gruppo Nazionale per l'Analisi Matematica, la Probabilit\'a e le loro Applicazioni (GNAMPA) of the Istituto Nazionale di Alta Matematica (INdAM).

\end{document}